\documentclass[11pt]{amsart}
\usepackage[utf8]{inputenc}
\usepackage{amssymb,amsmath}
\usepackage{graphicx}
\usepackage{color}
\usepackage{tikz}
\usepackage{pgfplots}
\usepackage{enumitem}
\usepackage{algorithm}
\usepackage{setspace}
\usepackage{algpseudocode}
\usepackage{mathtools}
\usepackage[hidelinks]{hyperref}
\usepackage{xfrac}

\usepackage{accents}
\usepackage{multicol} 
\usepackage{soul}
\usepackage{subcaption}
\newtheorem{theorem}{Theorem}

\theoremstyle{definition}

\newtheorem{example}[theorem]{Example}
\theoremstyle{lemma}
\newtheorem{lemma}[theorem]{Lemma}

\theoremstyle{remark}
\newtheorem{remark}[theorem]{Remark}

\newtheorem{assumption}[theorem]{Assumption}
\numberwithin{theorem}{section}
\numberwithin{equation}{section}
\numberwithin{table}{section}
\numberwithin{figure}{section}
\def\Vo{\V_{\ker}}
\def\calVker{\Vo}
\def\Vc{\V_\text{c}}

\def\cHo{\cH_{\ker}}

\def\calAo{\calA_{\ker}}
\def\uo{u_\text{ker}}
\def\uodot{\dot u_\text{ker}}

\def\uc{u_\text{c}}
\def\ucdot{\dot u_\text{c}}

\def\uon{u_{\text{ker},n}}
\def\uonn{u_{\text{ker},n+1}}
\def\uoEuler{u_{\text{ker},n+1}^\text{Eul}}

\def\I{\mathbb{I}}
\def\R{\mathbb{R}}
\def\A{\mathcal{A}}
\def\calA{\mathcal{A}}
 
\def\calB{\mathcal{B}}

\def\cH{\mathcal{H}}
\def\calH{\mathcal{H}}
\def\calI{\mathcal{I}}
\def\calK{\mathcal{K}}
\def\calL{\mathcal{L}}

\def\Q{\mathcal{Q}}
\def\calQ{\mathcal{Q}}

\def\V{\mathcal{V}}
\def\calV{\V}

\def\eps{\varepsilon}

\definecolor{blau}{RGB}{0, 51, 255}
\definecolor{hellblau}{RGB}{153, 204, 255}
\definecolor{hellrot}{RGB}{255, 0, 0}
\definecolor{firebrick}{RGB}{176, 34, 34} 
\definecolor{deep_pink}{RGB}{255, 20, 147} 
\definecolor{sky_blue}{RGB}{74, 112, 139}
\definecolor{slate_blue}{RGB}{71, 60, 139}
\definecolor{chartreuse}{RGB}{118, 238, 0}
\definecolor{chartreuseL}{RGB}{228, 255, 150}
\definecolor{light_blue}{RGB}{178, 223, 238}
\definecolor{dodge_blue}{RGB}{17, 78, 138}
\definecolor{code_backg}{RGB}{238, 216, 174}
\definecolor{myBlue1}{RGB}{101,149,239}  
\definecolor{myBlue2}{RGB}{113,104,238} 
\definecolor{myBlue3}{RGB}{30,144,255} 
\definecolor{myGreen1}{RGB}{154,204,50} 
\definecolor{myGreen2}{RGB}{69,169,0} 
\definecolor{myGreen3}{RGB}{154,205,50} 
\definecolor{myGreen4}{RGB}{105,139,34} 
\definecolor{myRed1}{RGB}{210,105,30} 
\definecolor{myRed2}{RGB}{165,42,42} 
\definecolor{myRed3}{RGB}{139,26,26} 
\definecolor{myLGray}{RGB}{225,225,225} 

\DeclareMathOperator{\id}{id}

\DeclareMathOperator{\image}{im}

\newcommand{\ds}{\ensuremath{\, \mathrm{d}s }}
\newcommand{\dt}{\ensuremath{\, \mathrm{d}t }}
\newcommand{\deta}{\ensuremath{\, \mathrm{d}\eta }}

\newcommand{\ddt}[1][t]{\ensuremath{\frac{\mathrm{d}}{\mathrm{d}#1}}}
\newcommand{\ddts}[1][t]{\ensuremath{\tfrac{\mathrm{d}}{\mathrm{d}#1}}}
\newcommand{\ddss}[1][s]{\ensuremath{\tfrac{\mathrm{d}}{\mathrm{d}#1}}}
\newcommand{\hook}{\ensuremath{\hookrightarrow}}
\newcommand{\Du}{\triangle u}
\newcommand{\Dl}{\triangle \lambda}
\allowdisplaybreaks
\begin{document}
\title[Exponential Integrators for Semi-linear Parabolic PDAEs]{Exponential Integrators for Semi-linear Parabolic Problems with Linear Constraints}
\author[]{R.~Altmann$^\dagger$, C.~Zimmer$^{\ddagger}$}
\address{${}^{\dagger}$ Department of Mathematics, University of Augsburg, Universit\"atsstr.~14, 86159 Augsburg, Germany}
\address{${}^{\ddagger}$ Institute of Mathematics MA\,{}4-5, Technical University Berlin, Stra\ss e des 17.~Juni 136, 10623 Berlin, Germany}
\email{robert.altmann@math.uni-augsburg.de, zimmer@math.tu-berlin.de}
\thanks{C.~Zimmer acknowledges support by the Deutsche Forschungsgemeinschaft (DFG, German Research Foundation) within the SFB 910, project number 163436311.}
\date{\today}
\begin{abstract}
This paper is devoted to the construction of exponential integrators of first and second order for the time discretization of constrained parabolic systems. For this extend, we combine well-known exponential integrators for unconstrained systems with the solution of certain saddle point problems in order to meet the constraints throughout the integration process. The result is a novel class of semi-explicit time integration schemes. 
We prove the expected convergence rates and illustrate the performance on two numerical examples including a parabolic equation with nonlinear dynamic boundary conditions. 
\end{abstract}
%
%
\maketitle
%
%
%
%
\section{Introduction}
Exponential integrators provide a powerful tool for the time integration of stiff ordinary differential equations as well as parabolic partial differential equations (PDE), cf.~\cite{Cer60, Law67, HocO10}. Such integrators are based on the possibility to solve the linear part -- which is responsible for the stiffness of the system -- in an exact manner. As a result, large time steps are possible which makes the method well-suited for time stepping, especially for parabolic systems where CFL conditions may be very restrictive. 
%
For semi-linear ODEs and parabolic PDEs exponential integrators are well-studied in the literature. This includes explicit and implicit exponential Runge-Kutta methods~\cite{CoxM02,HocO05,HocO05b}, exponential Runge-Kutta methods of high order~\cite{LuaO14}, exponential Rosenbrock-type methods~\cite{HocOS09}, and multistep exponential integrators~\cite{CalP06}. \medskip

In this paper, we construct and analyze exponential integrators for parabolic PDEs which underlie an additional (linear) constraint. This means that we aim to approximate the solution to 
\[
 \dot u(t) + \calA u(t)
 = f(t, u)
\]
which at the same time satisfies a constraint of the form $\calB u(t) = g(t)$. Such systems can be considered as differential-algebraic equations (DAEs) in Banach spaces, also called partial differential-algebraic equations (PDAEs), cf.~\cite{EmmM13, LamMT13, Alt15}. PDAEs of parabolic type include the transient Stokes problem (where $\calB$ equals the divergence operator) as well as problems with nontrivial boundary conditions (with $\calB$ being the trace operator). On the other hand, PDAEs of hyperbolic type appear in the modeling of gas and water networks~\cite{JanT14, EggKLMM18, AltZ18ppt} and in elastic multibody modeling~\cite{Sim98}. 

To the best of our knowledge, exponential integrators have not been considered for PDAEs so far. In the finite-dimensional case, however, exponential integrators have been analyzed for DAEs of (differential) index~1~\cite{HocLS98}. We emphasize that the parabolic PDAEs within this paper generalize index-2 DAEs in the sense that a standard spatial discretization by finite elements leads to DAEs of index 2. 
Known time stepping methods for the here considered parabolic PDAEs include splitting methods~\cite{AltO17}, algebraically stable Runge-Kutta methods~\cite{AltZ18}, and discontinuous Galerkin methods~\cite{VouR18}. 

In the first part of the paper we discuss the existence and uniqueness of solutions for semi-linear PDAEs of parabolic type with linear constraints. Afterwards, we propose two exponential integrators of first and second order for such systems. The construction of this novel class of time integration schemes benefits from the interplay of well-known time integration schemes for unconstrained systems and stationary saddle point problems in order to meet the constraints. Since the latter is done in an implicit manner, the combination with explicit schemes for the dynamical part of the system leads to so-called {\em semi-explicit} time integration schemes. As exponential integrators are based on the exact evaluation of semigroups, we need to extend this to the constrained case. The proper equivalent is the solution of a homogeneous but transient saddle point problem, which is a linear PDAE. 

The resulting exponential Euler scheme requires the solution of three stationary and a single transient saddle point problem in each time step. All these systems are linear, require in total only one evaluation of the nonlinear function, and do not call for another linearization step. Further, the transient system is homogeneous such that it can be solved without an additional regularization (or index reduction in the finite-dimensional case). 
The corresponding second-order scheme requires the solution of additional saddle point problems. Nevertheless, all these systems are linear and easy to solve. In a similar manner -- but under additional regularity assumptions -- one may translate more general exponential Runge-Kutta schemes to the constrained case. Here, however, we restrict ourselves to schemes of first and second order.  \medskip

The paper is organized as follows. In Section~\ref{sec:prelim} we recall the most important properties of exponential integrators for parabolic problems in the unconstrained case. Further, we introduce the here considered parabolic PDAEs, summarize all needed assumptions, and analyze the existence of solutions.  
The exponential Euler method is then subject of Section~\ref{sec:Euler}. Here we discuss two approaches to tackle the occurrence of constraints and prove first-order convergence. An exponential integrator of second order is then introduced and analyzed in Section~\ref{sec:secondOrd}. Depending on the nonlinearity, this scheme converges with order $2$ or reduced order~$\sfrac 32$. Comments on the efficient computation and numerical experiments for semi-linear parabolic systems illustrating the obtained convergence results are presented in Section~\ref{sec:numerics}.
%
%
\section{Preliminaries}\label{sec:prelim}
In this preliminary section we recall basic properties of exponential integrators when applied to PDEs of parabolic type. For this (and the later analysis) we consider the well-known~$\varphi$ functions. Further, we introduce the precise setting for the here considered parabolic systems with constraints and discuss their solvability. 
%
%
\subsection{Exponential integrators for parabolic problems}
As exponential integrators are based on the exact solution of linear homogeneous problems, we consider the recursively defined~$\varphi$-functions, see, e.g.~\cite[Ch.~11.1]{StrWP12}, 
\begin{align}
\label{eqn:phifunctions}
  \varphi_0(z) := e^z, \qquad
  \varphi_{k+1}(z) := \frac{\varphi_k(z) - \varphi_k(0)}{z}. 
\end{align}
For $z=0$ the values are given by~$\varphi_k(0) = 1/k!$. The importance of the~$\varphi$-functions comes from the fact that they can be equivalently written as integrals of certain exponentials. More precisely, we have for $k\ge 1$ that 
\[
  \varphi_{k}(z) 
  = \int_0^1 e^{(1-s)z} \frac{s^{k-1}}{(k-1)!} \ds. 
\]
We will consider these functions in combination with differential operators. For a bounded and invertible operator~$\calA\colon X\to X$ where $e^{t\calA} := \exp(t\calA)$ is well-defined, we can directly use the formulae in~\eqref{eqn:phifunctions} using the notion~$\frac 1\calA = \calA^{-1}$. 
As a result, the exact solution of a linear abstract ODE with polynomial right-hand side can be expressed in terms of~$\varphi_k$. More precisely, the solution of 
\begin{align}
\label{eqn:abstractODE}
  \dot u(t) + \calA u(t) 
  = \sum_{k=1}^n \frac{f_k}{(k-1)!}\, t^{k-1} \in X
\end{align}
with initial condition $u(0) = u_0$ and coefficients~$f_k\in X$ is given by
\begin{align}
\label{eqn:abstractODE:sln}
  u(t) 
  = \varphi_0(-t\calA)\, u_0 +  \sum_{k=1}^n \varphi_k(-t\calA)\, f_k\, t^k.
\end{align}
%
If $-\calA\colon D(\calA)\subset X \to X$ is an unbounded differential operator which generates a strongly continuous semigroup, then we obtain the following major property for the corresponding~$\varphi$-functions.  
\begin{theorem}[cf.~{\cite[Lem.~2.4]{HocO10}}]
Assume that the linear operator~$-\calA$ is the infinitesimal generator of a strongly continuous semigroup~$e^{-t\calA}$. Then, the operators~$\varphi_k(-\tau\calA)$ are linear and bounded in $X$. 
\end{theorem}
With the interpretation of the exponential as the corresponding semigroup, the solution formula for bounded operators~\eqref{eqn:abstractODE:sln} directly translates to linear parabolic PDEs of the form~\eqref{eqn:abstractODE} with an unbounded differential operator~$\calA$, cf.~\cite{HocO10}. 

The construction of exponential integrators for $\dot u(t) + \calA u(t) =f(t,u)$ is now based on the replacement of the nonlinearity~$f$ by a polynomial and~\eqref{eqn:abstractODE:sln}. Considering the interpolation polynomial of degree 0, i.e., evaluating the nonlinearity only in the starting value of $u$, we obtain the {\em exponential Euler scheme}. The corresponding scheme for constrained systems is discussed in Section~\ref{sec:Euler} and a second-order scheme in Section~\ref{sec:secondOrd}.
%
%
\subsection{Parabolic problems with constraints}\label{sec:prelim:PDAE}
In this subsection, we introduce the constrained parabolic systems of interest and gather assumptions on the involved operators. Throughout this paper we consider semi-explicit and semi-linear systems meaning that the constraints are linear and that the nonlinearity only appears in the low-order terms of the dynamic equation. Thus, we consider the following parabolic PDAE: 
find~$u\colon[0,T] \to \V$ and~$\lambda\colon[0,T] \to \Q$ such that   
\begin{subequations}
\label{eqn:PDAE}
\begin{alignat}{5}
	\dot{u}(t)&\ +\ &\calA u(t)&\ +\ &\calB^\ast \lambda(t)\, &= f(t, u) &&\qquad\text{in } \V^* \label{eqn:PDAE:a},\\
	& &\calB u(t)& & &= g(t) &&\qquad\text{in } \Q^*.
	\label{eqn:PDAE:b}
\end{alignat}
\end{subequations}
Therein, $\V$ and $\Q$ denote Hilbert spaces with respective duals~$\V^*$ and $\Q^*$. The space $\V$ is part of a Gelfand triple~$\V, \cH, \V^*$, cf.~\cite[Ch.~23.4]{Zei90a}. This means that~$\V$ is continuously (and densely) embedded in the pivot space~$\cH$ which implies~$\cH^* \hook \V^*$, i.e., the continuous embedding of the corresponding dual spaces. In this setting, the Hilbert space~$\cH$ is the natural space for the initial data. Note, however, that the initial condition may underlie a consistency condition due to the constraint~\eqref{eqn:PDAE:b}, cf.~\cite{EmmM13}. For the here considered analysis we assume slightly more regularity, namely~$u(0)=u_0 \in\V$, and consistency of the form~$\calB u_0 = g(0)$. 

The assumptions on the operators $\calA \in \calL(\V,\V^*)$ and $\calB \in \calL(\V,\Q^*)$ are summarized in the following. 
%
\begin{assumption}[constraint operator $\calB$]
\label{assB}
The operator~$\calB\colon \V \to \Q^*$ is linear, continuous, and satisfies an inf-sup condition, i.e., there exists a constant~$\beta>0$ such that 
\[
\adjustlimits \inf_{q\in\Q\setminus\{ 0 \} }\sup_{v\in\V\setminus\{0\}}
\frac{\langle \calB v, q\rangle}{\Vert v\Vert_\V \Vert q\Vert_\Q}
\ge \beta. 
\]
\end{assumption}
\begin{assumption}[differential operator $\calA$]
\label{assA}
The linear and continuous operator~$\calA\colon \V\to\V^*$ has the form~$\calA = \calA_1+\calA_2$ with~$\calA_1 \in \calL(\V,\V^*)$ being self-adjoint and~$\calA_2\in \calL(\V,\cH)$. Further, we assume that~$\calA$ is elliptic on~$\Vo := \ker\calB$, i.e., on the kernel of the constraint operator. 
\end{assumption}
Without loss of generality, we may assume under Assumption~\ref{assA} that~$\calA_1$ is elliptic on~$\Vo$. This can be seen as follows: With~$\mu_{\calA}$ denoting the ellipticity constant of $\calA$ and~$c_{\calA_2}$ the continuity constant of $\calA_2$, we set 
\[
  \calA_1 \leftarrow \calA_1 + \frac{c_{\calA_2}^2}{2 \mu_{\calA}} \id_{\calH}
  \quad\text{and}\quad
  \calA_2 \leftarrow \calA_2 - \frac{c_{\calA_2}^2}{2 \mu_{\calA}} \id_{\calH}.
\]
This then implies 
\begin{align*}
\langle \calA_1 v_{\ker},v_{\ker}\rangle \geq \mu_{\calA}\|v_{\ker}\|^2_{\calV} - c_{\calA_2} \|v_{\ker}\|_{\calV} \|v_{\ker}\|_{\calH} + \frac{c_{\calA_2}^2}{2 \mu_{\calA}} \|v_{\ker}\|^2_{\calH} \geq \frac{\mu_{\calA}}{2}\|v_{\ker}\|^2_{\calV}
\end{align*}
for all $v_{\ker} \in \calVker$. Hence, we assume throughout this paper that, given Assumption~\ref{assA}, $\calA_1$ is elliptic on~$\Vo$. As a result, $\calA_1$ induces a norm which is equivalent to the $\V$-norm on~$\Vo$, i.e., 
\begin{equation}
\label{eqn:norm_equivalence}
  \mu\, \|v_{\ker}\|^2_{\calV} 
  \leq \|v_{\ker}\|^2_{\calA_1} \leq 
  C\, \|v_{\ker}\|^2_{\calV}.
\end{equation}
\begin{remark}
The results of this paper can be extended to the case where $\calA$ only satisfies a G\aa{}rding inequality on~$\Vo$. In this case, we add to $\calA$ the term $\kappa\, \id_{\calH}$ such that $\calA + \kappa\, \id_{\calH}$ is elliptic on~$\Vo$ and add it accordingly to the nonlinearity~$f$. 
\end{remark}
Assumption~\ref{assB} implies that~$\calB$ is onto such that there exists a right-inverse denoted by~$\calB^-\colon\Q^*\to\V$. This in turn motivates the decomposition
\[
  \V = \Vo \oplus \Vc \qquad\text{with}\qquad
  \Vo = \ker\calB,\quad 
  \Vc = \image\calB^-.
\]   
We emphasize that the choice of the right-inverse (and respectively~$\Vc$) is, in general, not unique and allows a certain freedom in the modeling process. Within this paper, we define the complementary space $\Vc$ as in~\cite{AltZ18} in terms of the annihilator of $\Vo$, i.e., 
\[
  \Vc 
  := \{ v\in\V\, |\, \calA v \in \Vo^0 \}
  = \{ v\in\V\, |\, \langle \calA v, w\rangle = 0 \text{ for all } w\in \Vo \}.
\]

The analysis of constrained systems such as~\eqref{eqn:PDAE} is heavily based on the mentioned decomposition of~$\V$. Furthermore, we need the restriction of the differential operator to the kernel of~$\calB$, i.e., 
\[
  \calAo := \calA|_{\Vo} \colon \Vo\to\Vo^* := (\Vo)^*. 
\]
Note that we use here the fact that functionals in~$\V^*$ define functionals in~$\Vo^*$ simply through the restriction to~$\Vo$. The closure of~$\Vo$ in the~$\cH$-norm is denoted by~$\cHo := \overline{\Vo}^\cH$. Assumption~\ref{assA} now states that~$\calAo$ is an elliptic operator. This in turn implies that~$-\calAo$ generates an analytic semigroup on $\cHo$, see \cite[Ch.~7, Th.~2.7]{Paz83}. 

Finally, we need assumptions on the nonlinearity~$f$. Here, we require certain smoothness properties such as local Lipschitz continuity in the second component. The precise assumptions will be given in the respective theorems. 
\begin{example}
\label{exp:dynBC}
The (weak) formulation of semi-linear parabolic equations with dynamical (or Wentzell) boundary conditions~\cite{SprW10} fit into the given framework. For this, the system needs to be formulated as a coupled system which leads to the PDAE structure~\eqref{eqn:PDAE}, cf.~\cite{Alt19}. We emphasize that also the boundary condition may include nonlinear reaction terms. 
We will consider this example in the numerical experiments of Section~\ref{sec:numerics}. 
\end{example}
%
%
\subsection{Existence of solutions}
In this section we discuss the existence of solutions to~\eqref{eqn:PDAE}, where we use the notion of Sobolev-Bochner spaces~$L^2(0,T;X)$ and~$H^1(0,T;X)$ for a Banach space $X$, cf.~\cite[Ch.~23]{Zei90a}. For the case that $f$ is independent of~$u$, the existence of solutions is well-studied, see~\cite{Tar06,EmmM13,Alt15}. We recall the corresponding result in the special of $\calA$ being self-adjoint, which is needed in later proofs.
\begin{lemma}
\label{lem:PDAE_f_independent}
Let $\calA\in \calL(\calV,\calV^\ast)$ be self-adoint and elliptic on $\calVker$ and let $\calB$ satisfy Assumption~\ref{assB}. Further, assume $f \in L^2(0,T;\calH)$, $g \in H^1(0,T;\calQ^\ast)$, and~$u_0 \in \calV$ with $\calB u_0=g(0)$. Then, the PDAE~\eqref{eqn:PDAE} with right-hand sides $f$, $g$ -- independent of~$u$ -- and initial value $u_0$ has a unique solution
\begin{align*}
 u&\in C(0,T;\calV) \cap H^1(0,T;\calH), &\lambda &\in L^2(0,T;\calQ)
\end{align*}
with $u(0)=u^0$. The solution depends continuously on the data and satisfies 
\begin{equation}
\label{eqn:PDAE_f_independent}
 \|u(t) -\calB^- g(t)\|_{\calA}^2 
 \ \leq\ \|u_0-\calB^- g(0)\|_{\calA}^2 + \int_0^t \|f(s) - \calB^- \dot g(s)\|^2_{\calH} \ds.
\end{equation}
\end{lemma}
\begin{proof}
A sketch of the proof can be found in~\cite[Lem.~21.1]{Tar06}. For more details we refer to~\cite[Ch.~3.1.2.2]{Zim15}.
\end{proof}
In order to transfer the results of Lemma~\ref{lem:PDAE_f_independent} to the semi-linear PDAE~\eqref{eqn:PDAE} we need to reinterpret the nonlinearity $f\colon [0,T]\times \calV \to \calH$ as a function which maps an abstract measurable function $u\colon [0,T] \to \calV$ to  
$f(\,\cdot\,,u(\,\cdot\,))\colon [0,T] \to \calH$. For this, we need the classical Carath\'e{}odory condition, see \cite[Rem.~1]{GolKT92}, i.e.,  
\begin{enumerate}[label=\roman*.)]
	\item $v \mapsto f(t, v)$ is a continuous function for almost all $t \in [0,T]$, \label{item:Caratheodory_1}
	\item $t \mapsto f(t, v)$ is a measurable function for all $v\in \calV$.\label{item:Caratheodory_2}
\end{enumerate}
Furthermore, we need a boundedness condition such that the Nemytskii map induced by~$f$ maps $C([0,T];\calV)$ to $L^2(0,T;\calH)$. We will assume in the following that there exists a function~$k \in L^2(0,T)$ such that 
\begin{equation}
\label{eqn:f_bound}
\|f(t,v)\|_{\calH} \leq k(t)(1+\|v\|_{\calV})
\end{equation}
for all $v \in \calV$ and almost all $t\in [0,T]$. We emphasize that condition~\eqref{eqn:f_bound} is sufficient but not necessary for~$f$ to induce a Nemytskii map, cf.~\cite[Th.~1(ii)]{GolKT92}. We will use this condition to prove the existence and uniqueness of a global solution to~\eqref{eqn:PDAE}.

\begin{theorem}
\label{thm:existence_uniqueness_PDAE}
Assume that $\calA$ and $\calB$ satisfy Assumptions~\ref{assB} and~\ref{assA}. Let~$g \in H^1(0,T;\calQ^\ast)$ and suppose that $f\colon [0,T] \times \calV \to \calH$ satisfies the Carath\'e{}odory conditions as well as the uniform bound~\eqref{eqn:f_bound}. In addition, for every $v\in \calV$ there exists an open ball~$B_r(v) \subseteq \calV$ with radius $r=r(v)>0$ and a constant $L=L(v)\geq 0$, such that for almost every $t\in [0,T]$ it holds that
\begin{equation}
\label{eqn:f_lipschitz}
\|f(t,v_1) - f(t,v_2)\|_{\calH} 
\leq L\, \|v_1 -v_2\|_{\calV}
\end{equation}
for all $v_1,v_2 \in B_r(v)$. Then, for a consistent initial value $u_0 \in \calV$, i.e., $\calB u_0 = g(0)$, the semi-linear PDAE~\eqref{eqn:PDAE} has a unique solution
\begin{align*}
 u&\in C(0,T;\calV) \cap H^1(0,T;\calH), &\lambda &\in L^2(0,T;\calQ)
\end{align*}
with $u(0)=u^0$.
\end{theorem}
\begin{proof}
Without loss of generality, we assume that $\calA = \calA_1$. For this, we redefine $f(t,v) \leftarrow f(t,v)-\calA_2 v$, leading to an update of the involved constants $L\leftarrow L+c_{\calA_2}$ and $k \leftarrow k + c_{\calA_2}$ but leaving the radius $r$ unchanged. 


To prove the statement we follow the steps of~\cite[Ch.~6.3]{Paz83}. Let $t^\prime\in (0,T]$ be arbitrary but fixed. With~\eqref{eqn:f_bound} we notice that the Nemyskii map induced by~$f$ maps $C(0,t^\prime;\calV)$ into $L^2(0,t^\prime;\calH)$, cf.~\cite[Th.~1]{GolKT92}. Therefore, the solution map $S_{t^\prime} \colon C(0,t^\prime;\calV) \to C(0,t^\prime;\calV)$, which maps $y \in C(0,t^\prime;\calV)$ to the solution of
\begin{subequations}
\label{eqn:existence_uniqueness_PDAE_help}
\begin{alignat}{5}
	\dot{u}(t)&\ +\ &\calA u(t)&\ +\ &\calB^\ast \lambda(t)\, &= f(t, y(t)) &&\qquad\text{in } \V^*, \label{eqn:existence_uniqueness_PDAE_help_a}\\
	& &\calB u(t)& & &= g(t) &&\qquad\text{in } \Q^* \label{eqn:existence_uniqueness_PDAE_help_b}
\end{alignat}
\end{subequations}
with initial value $u_0$, is well-defined, cf.~Lemma~\ref{lem:PDAE_f_independent}. To find a solution to~\eqref{eqn:PDAE} we have to look for a unique fixed point of $S_{t^\prime}$ and show that $t^\prime$ can be extended to $T$. 

Let $\tilde{u}\in C(0,T;\calV)$ be the solution of the PDAE~\eqref{eqn:PDAE} for $f\equiv 0$ and initial value~$u_0$. With $r=r(u_0)$ and $L = L(u_0)$ we now choose $t_1 \in (0,T]$ such that \noindent
\begin{align}
  \|\tilde{u}(t)-u_0\|_{\calV} 
  &\ \leq\ \frac r2, \label{item:tildeu}\tag{a}\\
  \int_0^{t} |k|^2\ds 
  &\ \leq\ \frac{\mu r^2}{4\, (1+r + \|u_0\|_{\calV})^2},\label{item:bound}\tag{b}\\
  L^2t_1 
  &\ <\ \mu,\label{item:lipschitz} \tag{c}\\
  \int_0^{t} \tfrac{3}{\mu}\, |k|^2(1 + \|\tilde{u}\|_{\calV}^2)\ds  
  &\ \leq\ \frac{r^2}{4} \cdot \exp\big(-\tfrac{3}{\mu}\int_0^{t} |k|^2\ds \big) \label{item:gronwall}\tag{d}
\end{align}
for all $t\in [0,t_1]$. This is well-defined, since $\tilde{u}-u_0$ and the integrals in~\eqref{item:bound} and~\eqref{item:gronwall} are continuous functions in $t$, which vanish for $t=0$. We define 
\begin{equation*}
D := \big\{ y \in C(0,t_1;\calV)\,|\, \|y-\tilde{u}\|_{C([0,t_1],\calV)} \leq r/2\big\}
\end{equation*}
and consider $y_1, y_2 \in D$. By~\eqref{item:tildeu} we have $\|y_i - u_0\|_{C([0,t_1],\calV)} \leq r$. Using that $\tilde u$ and $S_{t_1} y_i$ satisfy the constraint~\eqref{eqn:existence_uniqueness_PDAE_help_b}, we obtain the estimate 
\begin{align*}
\mu\, \| (S_{t_1}y_i - \widetilde{u})(t)\|^2_{\calV} &\overset{\eqref{eqn:PDAE_f_independent}}{\leq} \!\int_0^t \|f(s,y_i(s))\|^2_{\calH} \ds\!\\ &\overset{\eqref{eqn:f_bound}}{\leq}\! \int_0^t |k(s)|^2\big(1+\|y_i(s)-u_0\|_{\calV} + \|u_0\|_{\calV}\big)^2\ds\\
&\overset{\hphantom{\eqref{eqn:f_bound}}}{\leq}\! \big(1+r + \|u_0\|_{\calV}\big)^2 \int_0^t |k(s)|^2\ds
\end{align*}
which implies with~\eqref{item:bound} that $S_{t_1}$ maps $D$ into itself. Further, we have 
\begin{align*}
\mu\, \| (S_{t_1}y_1 - S_{t_1}y_2)(t)\|^2_{\calV} \overset{\eqref{eqn:PDAE_f_independent}}{\leq} \int_0^t \|f(s,y_1(s))-f(s,y_2(s))\|^2_{\calH} \ds  
\overset{\eqref{eqn:f_lipschitz}}{\leq} L^2 t_1 \| y_1 - y_2\|^2_{C(0,t_1;\calV)}
\end{align*}
for all $t\leq t_1$, $i=1,2$. Together with the previous estimate and~\eqref{item:lipschitz}, this shows that $S_{t_1}$ is a contraction on~$D$. Hence, there exists a unique fixed point $u \in D \subset C([0,t_1],\calV)$ of~$S_{t_1}$ by the Banach fixed point theorem~\cite[Th.~1.A]{Zei92}. On the other hand, for every fixed point~$u^\star=S_{t_1} u^\star$ in~$C([0,t_1],\calV)$, we have the estimate
\begin{equation*}
\mu\, \| (u^\star - \widetilde{u})(t)\|^2_{\calV}
= \mu\, \| (S_{t_1} u^\star - \widetilde{u})(t)\|^2_{\calV} 
\le \int_0^t |k(s)|^2\big(1+\|(u^\star-\tilde{u})(s)\|_{\calV} + \|\tilde{u}(s)\|_{\calV}\big)^2\ds.
\end{equation*}
Using $(a+b+c)^2 \leq 3\, (a^2 + b^2 + c^2)$ and Gronwall's inequality it follows that 
\begin{equation}
\label{eqn:estimate_u_minus_tildeu}
\| (u^\star - \widetilde{u})(t)\|^2_{\calV} 
\leq \int_0^t \tfrac{3}{\mu}\, |k(s)|^2 \big(1+\|\tilde{u}(s)\|^2_{\calV}\big) \ds \cdot \exp\Big(\tfrac{3}{\mu} \int_0^t |k(s)|^2\ds\Big)
\end{equation}
for every $t \leq t_1$. Because of~\eqref{item:gronwall}, this shows that $u^\star$ is an element of~$D$ and thus, $u^\star=u$.

By considering problem~\eqref{eqn:PDAE} iterativly from $[t_{i-1},T]$, $t_0:=0$, to $[t_{i},T]$ with consistent initial value~$u_0=u(t_{i})$, we can extend~$u$ uniquely on an interval $\I$ with $u \in C(\I;\calV)$ and $u = S_{t^\prime} u$ for every $t^\prime \in \I$. Note that either $\I=[0,T]$ or $\I=[0,T^\prime)$ with $T^\prime \leq T$. The second case is only possible if $\|u(t)\|_{\calV} \to \infty$ for~$t \to T^\prime$, otherwise we can extend~$u$ by starting at $T^\prime$. But, since the estimate~\eqref{eqn:estimate_u_minus_tildeu} also holds for $u=u^\star$ and $t<T^\prime$, we have in limit that $\|u(T^\prime)\|_{\calV} \leq \|u(T^\prime) - \tilde{u}(T^\prime)\|_{\calV}+\|\tilde{u}(T^\prime)\|_{\calV}$ is bounded. Therefore, $u=S_Tu\in C([0,T];\calV)$. Finally, the stated spaces for $u$ and $\lambda$ follow by Lemma~\ref{lem:PDAE_f_independent} with right-hand side $f=f(\,\cdot\,,u(\,\cdot\,)).$
\end{proof}
\begin{remark}
\label{rem:uniformLipschitz}
Under the given assumptions on~$f$ from Theorem~\ref{thm:existence_uniqueness_PDAE}, one finds a radius~$r_u>0$ and a Lipschitz constant $L_u\in [0,\infty)$, both based on the solution~$u$, such that~\eqref{eqn:f_lipschitz} holds for all $x,y \in B_{r_u}(u(s))$ with $L=L_u$ and arbitrary $s\in [0,T].$ With these uniform constants one can show that the mapping of the data~$u_0 \in \calV$ and $g\in H^1(0,T;\calQ)$ with $\calB u_0 =g(0)$ to the solution $(u,\lambda)$ is continuous. 
\end{remark}

\begin{remark}
It is possible to weaken the assumption~\eqref{eqn:f_bound} of Theorem~\ref{thm:existence_uniqueness_PDAE} to $\|f(t,v)\|_{\calH} \leq k(t)(1+\|v\|^p_{\calV})$ for an arbitrary~$p >1$. Under this assumption one can show the existence of a unique solution of~\eqref{eqn:PDAE}, which may only exists locally.
\end{remark}

\begin{remark}
The assumptions considered in~\cite[Ch.~6.3]{Paz83} are stronger then the one in Theorem~\ref{thm:existence_uniqueness_PDAE}. If these additional assumptions are satisfied, then the existence and uniqueness of a solution to~\eqref{eqn:PDAE} follows directly by Lemma~\ref{lem:PDAE_f_independent}, \cite[Ch.~6, Th.~3.1 \& 3.3]{Paz83}, and the fact that every self-adjoint, elliptic operator~$\calA \in \calL(\calV,\calV^\ast)$ has a unique invertible square root $\calA^{\sfrac{1}{2}} \in \calL(\calV,\calH)$ with $\langle \calA v_1, v_2 \rangle = (\calA^{\sfrac{1}{2}}v_1, \calA^{\sfrac{1}{2}} v_2)$ for all $v_1,v_2 \in \calV$. This can be proven by interpreting~$\calA$ as an (unbounded) operator $\mathbf A\colon D(\mathbf A) \subset \calH \to \calH$ with domain $D(\mathbf A) := \calA^{-1}\calH \subset \calV \hookrightarrow \calH$ and the results of \cite[Ch.~6, Th.~4 \& Ch.~10, Th.~1]{BirS87} and~\cite[Th.~6.8]{Paz83}. 
%
\end{remark}
%
\subsection{A solution formula for the linear case}\label{sec:prelim:slnformula}
In the linear case, the solution of~\eqref{eqn:PDAE} can be expressed by the variation-of-constants formula (Duhamel's principle), cf.~\cite{EmmM13}. In the semi-linear case, we consider the term~$f(t,u)$ as a right-hand side which leads to an implicit formula only. This, however, is still of value for the numerical analysis of time integration schemes. 

The solution formula is based on the decomposition~$u=\uo+\uc$ with~$\uo\colon[0,T] \to \Vo$ and~$\uc\colon[0,T] \to \Vc$. The latter is fully determined by the constraint~\eqref{eqn:PDAE:b}, namely~$\uc(t) = \calB^-g(t)\in\Vc$. For~$\uo$ we consider the restriction of~\eqref{eqn:PDAE:a} to the test space~$\Vo$. Since the Lagrange multiplier disappears in this case, we obtain 
\begin{align}
\label{eqn:uker}
  \uodot + \calAo \uo
  = \uodot + \calA \uo 
  = f(t, \uo+\uc) - \ucdot\qquad\text{in } \Vo^*.
\end{align}
Note that the right-hand side is well-defined as functional in~$\Vo^*$ using the trivial restriction of~$\V^*$ to~$\Vo^*$. Further, the term~$\calA \uc$ disappears under test functions in~$\Vo$ due to the definition of~$\Vc$. If this orthogonality is not respected within the implementation, then this term needs to be reconsidered. 

The solution to~\eqref{eqn:uker} can be obtained by an application of the variation-of-constants formula. Since the semigroup can only be applied to functions in $\cHo$, we introduce the operator  
\[
\iota_0\colon \cH \equiv \cH^* \to \cHo^* \equiv \cHo.
\]
This operator is again based on a simple restriction of test functions and leads to the solution formula  
\begin{align*}
  u(t) 
  &= \uc(t) + \uo(t) \\
  &= \calB^-g(t) + e^{-t\calAo}\uo(0) + \int_0^t e^{-(t-s)\calAo} \iota_0 \big[ f(s, \uo(s)+\uc(s)) - \ucdot(s) \big] \ds. 
\end{align*}
Assuming a partition of the time interval~$[0,T]$ by~$0=t_0 < t_1 < \dots < t_N = T$, we can write the solution formula in the form 
\begin{align}
  u(t_{n+1}) - &\calB^-g_{n+1} \notag \\
  &= e^{-(t_{n+1}-t_n)\calAo} \big[u(t_n)-\calB^-g_n\big] + \int_{t_n}^{t_{n+1}} e^{-(t_{n+1}-s)\calAo} \iota_0 \big[ f(s, u(s)) - \ucdot(s) \big] \ds. \label{eqn:uexact}
\end{align}
Note that we use here the abbreviation~$g_n := g(t_n)$. In the following two sections we construct exponential integrators for constrained semi-linear systems of the form~\eqref{eqn:PDAE}. Starting point is a first-order scheme based on the exponential Euler method applied to equation~\eqref{eqn:uker}.
%

\section{The Exponential Euler Scheme}\label{sec:Euler}
The idea of exponential integrators is to approximate the integral term in~\eqref{eqn:uexact} by an appropriate quadrature rule. Following the construction for PDEs~\cite{HocO10}, we consider in this section the function evaluation at the beginning of the interval. This then leads to the scheme 
\begin{align}
  u_{n+1} - \calB^-g_{n+1} 
  &= e^{-\tau \calAo}\big[ u_n - \calB^-g_{n}\big] + \int_0^\tau e^{-(\tau-s)\calAo} \iota_0 \big[ f(t_n, u_n) - \ucdot(t_n) \big] \ds \notag\\
  &= \varphi_0(-\tau \calAo)\big(u_n - \calB^-g_{n}\big) + \tau\varphi_1(-\tau\calAo)  \big( \iota_0 \big[ f(t_n, u_n) - \calB^- \dot g_n \big] \big). \label{eqn:expEuler}
\end{align}
As usual, $u_n$ denotes the approximation of $u(t_n)$. Further, we restrict ourselves to a uniform partition of~$[0,T]$ with step size~$\tau$ for simplicity. Assuming that the resulting approximation satisfies the constraint in every step, we have~$u_n - \calB^-g_{n} \in \Vo \hookrightarrow \cHo$ such that the semigroup~$e^{-\tau \calAo}$ is applicable. The derived formula~\eqref{eqn:expEuler} is beneficial for the numerical analysis but lacks the practical access which we tackle in the following. 
%
\subsection{The practical method}
Since the evaluation of the~$\varphi$-functions with the operator~$\calAo$ is not straightforward, we reformulate the method by a number of saddle point problems. Furthermore, we need evaluations of~$\calB^-$ applied to the right-hand side~$g$ (or its time derivative). Also this is replaced by the solution of a saddle point problem. 

Consider $x := \calB^-g_n = \calB^-g(t_{n}) \in \Vc \subseteq \V$. Then, $x$ can be written as the solution of the stationary auxiliary problem 
\begin{subequations}
\label{eqn:Binvers}	
\begin{alignat}{4}
  &\calA x&\ +\ &\calB^* \nu\ &=&\ 0  &&\qquad\text{in } \V^* , \label{eqn:Binvers:a} \\
  &\calB x& & &=&\ g_n &&\qquad\text{in } \Q^*. \label{eqn:Binvers:b}
\end{alignat}
\end{subequations}
Note that equation~\eqref{eqn:Binvers:b} enforces the connection of~$x$ to the right-hand side~$g$ whereas the first equation of the system guarantees the desired~$\calA$-orthogonality. The Lagrange multiplier~$\nu$ is not of particular interest and simply serves as a dummy variable. The unique solvability of system~\eqref{eqn:Binvers} is discussed in the following lemma. 
\begin{lemma}
\label{lem:stationary}
Let the operators $\calA$ and $\calB$ satisfy Assumptions~\ref{assB} and ~\ref{assA}. Then, for every $g_n\in \Q^*$ there exists a unique solution~$(x,\nu) \in \Vc\times\Q$ to system~\eqref{eqn:Binvers}.  
\end{lemma}
\begin{proof}
Under the given assumptions on the operators $\calA$ and $\calB$ there exists a unique solution~$(x,\nu) \in \V\times\Q$ to~\eqref{eqn:Binvers}, even in the case with an inhomogeneity in the first equation, see~\cite[Ch.~II, Prop.~1.3]{BreF91}. It remains to show that $x$ is en element of $\Vc$. For this, note that $x$ satisfies $\langle \calA x, w\rangle =0$ for all $w\in\Vo$, since the $\calB^*$-term vanishes for these test functions. This, however, is exactly the definition of the complement space~$\Vc$.
\end{proof}	
%
Being able to compute~$\calB^-g_n$, we are now interested in the solution of problems involving the operator~$\calAo$. This will be helpful for the reformulation of the exponential Euler method~\eqref{eqn:expEuler}. We introduce the auxiliary variable~$w_n\in \Vo$ as the solution of 
\[
  \calAo w_n 
  = f(t_n, u_n) - \ucdot(t_n) 
  = f(t_n, u_n) - \calB^- \dot g_n 
  \qquad\text{in } \Vo^*.
\]
This is again equivalent to a stationary saddle point problem, namely 
\begin{subequations}
\label{eqn:wn}	
\begin{alignat}{4}
&\calA w_n&\ +\ &\calB^\ast \nu_n\ &=&\ f(t_n, u_n) - \calB^- \dot g_n  &&\qquad\text{in } \V^* ,\\
&\calB w_n& & &=&\ 0 &&\qquad\text{in } \Q^*.
\end{alignat}
\end{subequations}
As above, the Lagrange multiplier is only introduced for a proper formulation and not of particular interest in the following. The unique solvability of system~\eqref{eqn:wn} follows again by Lemma~\ref{lem:stationary}, since the right-hand side of the first equation is an element of~$\V^*$. In order to rewrite~\eqref{eqn:expEuler}, we further note that the recursion formula for~$\varphi_1$ implies  
\[
  \tau\varphi_1(-\tau\calAo) h 
  = - \big[ \varphi_0(-\tau\calAo) - \id \big] \calAo^{-1} h
\] 
for all $h\in\cHo$. Recall that $\calAo$ is indeed invertible due to Assumption~\ref{assA}. Thus, the exponential Euler scheme can be rewritten as 
\begin{align*}
  u_{n+1} 
  = \calB^-g_{n+1} + \varphi_0(-\tau \calAo)\big(u_n - \calB^-g_{n} - w_n\big) + w_n.
\end{align*}
Finally, we need a way to compute the action of~$\varphi_0(-\tau \calAo)$. For this, we consider the corresponding PDAE formulation. The resulting method then reads~$u_{n+1} = \calB^-g_{n+1} + z(t_{n+1}) + w_n$, where $z$ is the solution of the {\em linear} homogeneous PDAE 
\begin{subequations}
\label{eqn:zn}	
\begin{alignat}{5}
  \dot{z}(t)&\ +\ &\calA z(t)&\ +\ &\calB^\ast \mu(t)\, &= 0 &&\qquad\text{in } \V^*,\\
  & &\calB z(t)& & &= 0 &&\qquad\text{in } \Q^*
\end{alignat}
\end{subequations}
with initial condition~$z(t_n) = u_n-\calB^-g_{n}-w_n$. Thus, the exponential Euler scheme given in~\eqref{eqn:expEuler} can be computed by a number of saddle point problems.  We summarize the necessary steps in Algorithm~\ref{alg:expEuler}.
\begin{algorithm}
\setstretch{1.1}
\caption{Exponential Euler scheme}
\label{alg:expEuler}
\begin{algorithmic}[1]
	\State {\bf Input}: step size~$\tau$, consistent initial data~$u_0\in\V$, right-hand sides $f$, $g$
	\vspace{0.5em}
	\For{$n=0$ {\bf to} $N-1$} 
	\State compute $\calB^-g_n$, $\calB^-g_{n+1}$, and $\calB^-\dot g_n = \calB^-\dot g(t_n)$ by~\eqref{eqn:Binvers} \label{alg:expEuler_Bminus}
	\State compute $w_n$ by~\eqref{eqn:wn} \label{alg:expEuler_w}
	\State compute $z$ as solution of~\eqref{eqn:zn} on~$[t_n,t_{n+1}]$ with initial data~$u_n-\calB^-g_{n}-w_n$ \label{alg:expEuler_z}
	\State set $u_{n+1} = \calB^-g_{n+1} + z(t_{n+1}) + w_n$
	\EndFor 	
\end{algorithmic}
\end{algorithm}
\begin{remark}
One step of the exponential Euler scheme consists of the solution of four (from the second step on only three) stationary and a single transient saddle point problem, including only one evaluation of the nonlinear function~$f$. We emphasize that all these systems are linear such that no Newton iteration is necessary in the solution process. Furthermore, the time-dependent system is homogeneous such that it can be solved without the need of a regularization. 
\end{remark}
%
\subsection{Convergence analysis}
In this section we analyze the convergence order of the exponential Euler method for constrained PDAEs of parabolic type. For the unconstrained case it is well-known that  the convergence order is one. Since our approach is based on the unconstrained PDE~\eqref{eqn:uker} of the dynamical part in $\calVker$, we expect the same order for the solution of Algorithm~\ref{alg:expEuler}. For the associated proof we will assume that the approximation~$u_n$ lies within a strip of radius~$r$ around~$u$, where~$f$ is locally Lipschitz continuous with constant~$L>0$. Note that by Remark~\ref{rem:uniformLipschitz} there exists such a uniform radius and local Lipschitz constant. Furthermore, a sufficiently small step size~$\tau$ guarantees that~$u_n$ stays within this strip around~$u$, since the solution~$z$ of~\eqref{eqn:zn} and $\calB^- g$ are continuous.
\begin{theorem}[Exponential Euler]
\label{thm:Euler}
Consider the assumptions of Theorem~\ref{thm:existence_uniqueness_PDAE} including Assumptions~\ref{assB} and \ref{assA}. Further, let the step size~$\tau$ be sufficiently small such that the derived approximation~$u_n$ lies within a strip along~$u$ in which $f$ is locally Lipschitz continuous with a uniform constant~$L>0$. For the right-hand side of the constraint we assume~$g\in H^{2}(0,T;\Q^*)$. If the exact solution of~\eqref{eqn:PDAE} satisfies~$\ddt f(\cdot,u(\cdot)) \in L^{2}(0,T;\cH)$, then the approximation~$u_n$ obtained by the exponential Euler scheme of Algorithm~\ref{alg:expEuler} satisfies 
\begin{equation*}
  \Vert u_n - u(t_n) \Vert_{\calV}^2
  \ \lesssim\ \tau^2 \int_0^{t_n} \|\ddts f(t, u(t))\|^2_{\calH} + \| \calB^- \ddot g(t)\|^2_{\calH} \dt.
\end{equation*}
Note that the involved constant only depends on $t_n$, $L$, and the operator~$\calA$.
\end{theorem}
\begin{proof}
With $w_n$ and $z$ from~\eqref{eqn:wn} and~\eqref{eqn:zn}, respectively, we define $U(t):=z(t)+w_n+\calB^-g(t)$ for $t\in [t_n,t_{n+1}]$, $n=0,\ldots,N-1$. This function satisfies 
\begin{equation*}
U(t_n)= z(t_n) + w_n + \calB^-g_n = u_n \quad \text{and} \quad U(t_{n+1})=z(t_{n+1}) + w_n + \calB^-g_{n+1} = u_{n+1}.
\end{equation*}
Furthermore, since $\dot U(t) = \dot z(t) + \calB^- \dot g(t)$, 
the function $U$ solves the PDAE 
\begin{alignat*}{5}
  \dot{U}(t)&\ +\ &\calA U(t)&\ +\ &\calB^\ast \Lambda(t)\, &= f(t_n,u_n) + \calB^-(\dot g(t) - \dot g_n) &&\qquad\text{in } \V^*,\\
  & &\calB U(t)& & &= g(t) &&\qquad\text{in } \Q^*
\end{alignat*}
on $[t_n,t_{n+1}]$, $n=0,\ldots,N-1$ with initial value $U(t_0) = u_0$. 
To shorten notation we define $\Du := u - U$ and $\Dl := \lambda - \Lambda$, which satisfy  
\begin{alignat*}{5}
  \ddts \Du &\ +\ &\calA_1 \Du&\ +\ &\calB^\ast \Dl\, &= f(\cdot,u(\cdot))- f(t_n,u_n) - \calA_2 \Du - \calB^-\big(\dot g - \dot g_n\big) &&\qquad\text{in } \V^*,\\
  & &\calB \Du & & &= 0  &&\qquad\text{in } \Q^*
\end{alignat*}
on each interval $[t_n,t_{n+1}]$ with initial value $\Du(t_0) = 0$ if $n=0$ and $\Du(t_n) = u(t_n) - u_n$ otherwise. In the following, we derive estimates of~$\Du$ on all sub-intervals. Starting with $n=0$, we have by Lemma~\ref{lem:PDAE_f_independent} that 
\begin{align*}
\| \Du(t)\|^2_{\calA_1}\!
&\overset{\eqref{eqn:PDAE_f_independent}}{\leq} \!\int_0^t \|f(s,u(s))- f(0,u_0) - \calA_2 \Du(s) - \calB^-\big(\dot g(s)- \dot g_0\big)\|^2_{\calH} \ds\\
&\overset{\hphantom{\eqref{eqn:PDAE_f_independent}}}{\leq} 2 \!\int_0^t \Big\|\int_0^s \ddts[\eta] f(\eta,u(\eta)) - \calB^-\ddot g(\eta) \deta\Big\|^2_{\calH}  + \tfrac{c^2_{\calA_2}}{\mu} \| \Du(s)\|_{\calA_1}^2 \ds.
\end{align*}
By Gronwall's lemma and $t=t_1=\tau$ we obtain with $c:=2c^2_{\calA_2}{\mu}^{-1}$ the bound
\begin{equation}
\label{eqn:estimate_Euler_0}
\begin{aligned}
\| u(t_1) - u_1\|^2_{\calA_1}
\leq\, &2\, e^{c\tau} \int_0^\tau \Big\|\int_0^s \ddts[\eta] f(\eta,u(\eta)) - \calB^-\ddot g(\eta) \deta\Big\|^2_{\calH} \ds\\
\leq\, &2\, e^{c\tau} \int_0^\tau s \int_0^s \| \ddts[\eta] f(\eta,u(\eta)) - \calB^-\ddot g(\eta) \|^2_{\calH} \deta \ds\\
\leq\, &2\, e^{c\tau}\tau^2  \underbrace{\int_0^\tau \| \ddss f(s,u(s))\|^2_{\calH} + \|\calB^-\ddot g(s)\|^2_{\calH} \ds }_{=:\,  \calI(\ddt f,\, \ddot g,\, 0,\, t_1)}.
\end{aligned}
\end{equation}
With the uniform Lipschitz constant $L$ we have for $n\ge 1$ that
\begin{align*}
\int_{t_n}^{t_{n+1}} \|f(s,u(s)&)-f(t_n,u_n)\|^2_{\calH} \ds\\
\leq\, &2 \int_{t_n}^{t_{n+1}} \|f(t_n,u(t_n))-f(t_n,u_n)\|^2_{\calH} + \|f(s,u(s))-f(t_n,u(t_n))\|^2_{\calH} \ds\\
\leq\, &2\, \tau \tfrac{L^2}{\mu}\, \| u(t_n) - u_n\|^2_{\calA_1} + 2\int_{t_n}^{t_{n+1}} (s-t_n)\int_{t_n}^s \|\ddts[\eta] f(\eta,u(\eta))\|^2_{\calH} \deta \ds. 
\end{align*}
With this, we obtain similarly as in~\eqref{eqn:estimate_Euler_0} and with Young's inequality, 
\begin{align}
\label{eqn:estimate_Euler_n}
\| u(t_{n+1}) - u_{n+1}\|^2_{\calA_1} 
\leq e^{c\tau} \Big[ (1+3\,\tau \tfrac{L^2}{\mu})\| u(t_{n}) - u_{n}\|^2_{\calA_1} 
+ 3\, \tau^2\, \calI(\ddts f,\ddot g,t_n,t_{n+1})\Big].
\end{align}
Therefore, with $(1+x) \leq e^x$, estimate~\eqref{eqn:estimate_Euler_0}, and an iterative  application of the estimate~\eqref{eqn:estimate_Euler_n} we get
\begin{align*}
\| u(t_{n+1}) - u_{n+1}\|^2_{\calA_1} 
&\leq \tau^2\ 3 \sum_{k=0}^n \exp( c \tau)^{n+1-k}  (1+3\, \tau \tfrac{L^2}{\mu})^{n-k}\, \calI(\ddts f,\ddot g,t_k,t_{k+1})\\
&\leq\tau^2\ 3\, \exp(c\, t_{n+1}) \exp\big( 3\tfrac{L^2}{\mu} t_{n}\big)\,  \calI(\ddts f,\ddot g,0,t_{n+1})
\end{align*}
for all~$n=0,\ldots,N-1$. The stated estimate finally follows by the equivalence of~$\|\cdot\|_{\calV}$ and~$\|\cdot\|_{\calA}$ on~$\calVker$, see~\eqref{eqn:norm_equivalence}.
\end{proof}
\begin{remark}
The assumption on the step size $\tau$ only depends on the nonlinearity~$f$ and not on the operator~$\calA$. Thus, this condition does not depend on the stiffness of the system and still allows large time steps.
\end{remark}
\begin{remark}
\label{rem:symA}
In the case of a self-adjoint operator~$\calA$, i.e., $\calA_2=0$, the convergence result can also be proven by the restriction to test functions in~$\Vo$ and the application of corresponding results for the unconstrained case, namely~\cite[Th.~2.14]{HocO10}. This requires similar assumptions but with~$\ddt f(\cdot,u(\cdot)) \in L^{\infty}(0,T;\cH)$.

We like to emphasize that this procedure is also applicable if~$\calA_2\neq 0$ by moving~$\calA_2$ into the nonlinearity~$f$. This, however, slightly changes the proposed scheme, since then only~$\calA_2 u_n$ enters the approximation instead of $\calA_2 u(t)$. In practical applications this would also require an additional effort in order to find the symmetric part of the differential operator~$\calA$ which is still elliptic on~$\Vo$.
\end{remark}
%
%
\subsection{An alternative approach}
A second approach to construct an exponential Euler scheme which is applicable to constrained systems is to formally apply the method to the corresponding singularly perturbed PDE. This approach was also considered in~\cite{HocLS98} for DAEs of index 1. In the present case, we add a small term~$\eps \dot\lambda$ into the second equation of~\eqref{eqn:PDAE}. Thus, we consider the system 
\begin{subequations}
\label{eqn:PDAE:eps}
	\begin{alignat}{5}
	\dot{u}(t)&\ +\ &\calA u(t)&\ +\ &\calB^\ast \lambda(t)\, &= f(t, u) &&\qquad\text{in } \V^* \label{eqn:PDAE:eps:a},\\
	\eps\dot{\lambda}(t)&\ +\ &\calB u(t)& & &= g(t) &&\qquad\text{in } \Q^*,
	\label{eqn:PDAE:eps:b}
	\end{alignat}
\end{subequations}
which can be written in operator matrix form as 
\[
  \begin{bmatrix} \dot u \\ \dot\lambda  \end{bmatrix}
  = \begin{bmatrix} \id &  \\ & \tfrac{1}{\eps}\id \end{bmatrix}  
  \Bigg\{
  -\begin{bmatrix} \calA & \calB^* \\ \calB & \end{bmatrix}  
  \begin{bmatrix} u \\ \lambda  \end{bmatrix} + 
  \begin{bmatrix} f(t,u) \\ g(t) \end{bmatrix}
  \Bigg\}.
\]
For this, an application of the exponential Euler method yields the scheme 
\[
  \begin{bmatrix} u_{n+1} \\ \lambda_{n+1} \end{bmatrix}
  = \varphi_0 \Big(-\tau \begin{bmatrix} \calA & \calB^* \\ \tfrac{1}{\eps}\calB & \end{bmatrix} \Big)
  \begin{bmatrix} u_n \\ \lambda_n \end{bmatrix} 
  + \tau \varphi_1\Big(-\tau \begin{bmatrix} \calA & \calB^* \\ \tfrac{1}{\eps}\calB & \end{bmatrix} \Big) 
  \begin{bmatrix} f(t_n,u_n) \\ \tfrac{1}{\eps} g_n \end{bmatrix}.
\]
We introduce the auxiliary variables~$(\bar w_n, \bar \nu_n) \in \V\times\Q$ as the unique solution to the stationary saddle point problem 
\begin{alignat*}{4}
  &\calA \bar w_n&\ +\ &\calB^\ast \bar\nu_n\ &=&\ f(t_n, u_n) &&\qquad\text{in } \V^* ,\\
  &\calB \bar w_n& & &=&\ \theta g_n + (1-\theta)g_{n+1}  &&\qquad\text{in } \Q^*.
\end{alignat*}
The included parameter~$\theta\in [0,1]$ controls the consistency as outlined below. Then, the exponential Euler method can be rewritten as 
\[
  \begin{bmatrix} u_{n+1} \\ \lambda_{n+1} \end{bmatrix}
%
%
%
  = \varphi_0 \Big(-\tau \begin{bmatrix} \calA & \calB^* \\ \tfrac{1}{\eps}\calB & \end{bmatrix} \Big)
  \begin{bmatrix} u_n -\bar w_n \\ \lambda_n - \bar\mu_n \end{bmatrix} + \begin{bmatrix}\bar w_n \\ \bar \mu_n \end{bmatrix},
\]
which allows an interpretation as the solution of a linear (homogeneous) PDE. Finally, we set~$\eps=0$, which leads to the following time integration scheme: Given~$\bar w_n$, solve on~$[t_n, t_{n+1}]$ the linear system  
\begin{alignat*}{5}
  \dot{z}(t)&\ +\ &\calA z(t)&\ +\ &\calB^\ast \mu(t)\, &= 0 &&\qquad\text{in } \V^*,\\
  & &\calB z(t)& & &= 0 &&\qquad\text{in } \Q^*
\end{alignat*}
with initial condition~$z(t_n) = u_n - \bar w_n$. The approximation of~$u(t_{n+1})$ is then defined through~$u_{n+1} := z(t_{n+1}) + \bar w_n$. 
 
We emphasize that the initial value of~$z$ may be inconsistent. In this case, the initial value needs to be projected to~$\cHo$, cf.~Section~\ref{sec:prelim:slnformula}. If the previous iterate satisfies~$\calB u_n = g_n$, then the choice $\theta=1$ yields~$\calB z(t_n) = 0$ and thus, consistency. This, however, does not imply~$\calB u_{n+1} = g_{n+1}$. On the other hand, $\theta=0$ causes an inconsistency for~$z$ but guarantees~$\calB u_{n+1} = g_{n+1}$. We now turn to an exponential integrator of higher order. 
\section{Exponential Integrators of Second Order}\label{sec:secondOrd}
This section is devoted to the construction of an exponential integrator of order two for constrained parabolic systems of the form~\eqref{eqn:PDAE}. In particular, we aim to transfer the method given in~\cite[Exp.~11.2.2]{StrWP12}, described by the {\em Butcher tableau}
\begin{equation}
\label{eqn:butcher_second_order}
\begin{array}{r|cc}
 	0 &  &  \\
 	1 & \varphi_1 & \\
 	\hline
 	 & \varphi_1-\varphi_2 & \varphi_2
 \end{array} 
\end{equation}
to the PDAE case. In the unconstrained case, i.e., for~$\dot v + \calAo v = \tilde f(t,v)$ in $\Vo^*$, one step of this method is defined through
\begin{subequations}
\label{eqn:secondOrder}
\begin{align}
  v_{n+1}^\text{Eul} &:= \varphi_0(-\tau\calAo)v_n + \tau\varphi_1(-\tau\calAo)\tilde f(t_n,v_n), \label{eqn:secondOrder:a} \\
  v_{n+1} &:= v_{n+1}^\text{Eul} + \tau\varphi_2(-\tau\calAo) \big[ \tilde f(t_{n+1},v_{n+1}^\text{Eul}) - \tilde f(t_n,v_n) \big]. \label{eqn:secondOrder:b}
\end{align}
\end{subequations}
Similarly as for the exponential Euler method, we will define a number of auxiliary problems in order to obtain an applicable method for parabolic systems with constraints.
%
\subsection{The practical method}\label{sec:secondOrd:method}
We translate the numerical scheme~\eqref{eqn:secondOrder} to the constrained case. Let~$u_n$ denote the given approximation of~$u(t_n)$. Then, the first step is to perform one step of the exponential Euler method, cf.~Algorithm~\ref{alg:expEuler}, leading to~$u_{n+1}^\text{Eul}$. Second, we compute~$w_n^\prime$ as the solution of the stationary problem  
\begin{subequations}
\label{eqn:second:wnA}	
\begin{alignat}{4}
	&\calA w_n'&\ +\ &\calB^\ast \nu_n'\ &=&\ f(t_{n+1},u_{n+1}^\text{Eul}) - \calB^- \dot g_{n+1} - f(t_n,u_n) + \calB^- \dot g_n  &&\qquad\text{in } \V^* ,\\
	&\calB w_n'& & &=&\ 0 &&\qquad\text{in } \Q^* 
\end{alignat}
\end{subequations}
and~$w_n''$ as the solution of 
\begin{subequations}
\label{eqn:second:wnB}	
\begin{alignat}{4}
	&\calA w_n''&\ +\ &\calB^\ast \nu_n''\ &=&\ \tfrac{1}{\tau} w_n' &&\qquad\text{in } \V^* ,\\
	&\calB w_n''& & &=&\ 0 &&\qquad\text{in } \Q^*.
\end{alignat}
\end{subequations}
Note that, due to the recursion formula~\eqref{eqn:phifunctions}, $w_n'$ and~$w_n''$ satisfy the identity 
\begin{align*}
  \tau\varphi_2(-\tau\calAo)\, \iota_0\big[ f(t_{n+1},u_{n+1}^\text{Eul}) - \calB^- \dot g_{n+1} - f(t_n,u_n) + \calB^- \dot g_n  \big] 
  &= - \varphi_1(-\tau\calAo) w_n' + w_n'  \\
  &= \varphi_0(-\tau\calAo)w_n'' - w_n'' + w_n'.   
\end{align*}
It remains to compute~$\varphi_0(-\tau\calAo)w_n''$ and thus, to solve a linear dynamical system with starting value~$w_n''$. More precisely, we consider the homogeneous system~\eqref{eqn:zn} on the time interval~$[t_n, t_{n+1}]$ with initial value $z(t_n) = w_n''$. The solution at time~$t_{n+1}$ then defines the new approximation by
\[
  u_{n+1} 
  := u^\text{Eul}_{n+1} + z(t_{n+1}) - w_n'' + w_n'. 
\] 
Note that the consistency is already guaranteed by the exponential Euler step which yields ~$\calB u_{n+1}  = \calB u^\text{Eul}_{n+1} = g_{n+1}$. The resulting exponential integrator is summarized in Algorithm~\ref{alg:secondOrd}. 
\begin{algorithm}
\setstretch{1.1}
\caption{A second-order exponential integrator}
\label{alg:secondOrd}
\begin{algorithmic}[1]
	\State {\bf Input}: step size~$\tau$, consistent initial data~$u_0\in\V$, right-hand sides $f$, $g$
	\vspace{0.5em}
	\For{$n=0$ {\bf to} $N-1$} 
	\State compute one step of the exponential Euler method for $u_n$ leading to $u^\text{Eul}_{n+1}$
	\State compute $\calB^-\dot g_n$ and $\calB^-\dot g_{n+1}$ by~\eqref{eqn:Binvers}
	\State compute~$w_n'$ by~\eqref{eqn:second:wnA} 
	\State compute~$w_n''$ by~\eqref{eqn:second:wnB} 
	\State compute $z$ as solution of~\eqref{eqn:zn} on~$[t_n,t_{n+1}]$ with initial condition~$z(t_n)=w_n''$ \label{alg:secondOrd:z}
	\State set $u_{n+1} = u^\text{Eul}_{n+1} + z(t_{n+1}) - w_n'' + w_n'$
	\EndFor 	
\end{algorithmic}
\end{algorithm}
%
%
\subsection{Convergence analysis}\label{sec:secondOrd:convergence}
In this subsection we aim to prove the second-order convergence of Algorithm~\ref{alg:secondOrd} when applied to parabolic PDAEs of the form~\eqref{eqn:PDAE}. For this, we examine two cases. 
First, we consider a nonlinearity with values in $\V$, i.e., we assume~$f\colon [0, T]\times\V\to\V$. Further, we assume~$\calA$ to be self-adjoint, meaning that~ $\calA_2=0$. Note that this may be extended to general~$\calA$ as mentioned in Remark~\ref{rem:symA}. In this case, the convergence analysis is based on the corresponding results for unconstrained systems. 
Second, we consider the more general case with nonlinearities~$f\colon [0, T]\times\V\to\cH$. Here, it can be observed that the convergence order drops to~$\sfrac 3 2$. Note, however, that this already happens in the pure PDE case. 
\begin{theorem}[Second-order scheme]
\label{thm:SecondOrdr}
In the setting of Section~\ref{sec:prelim:PDAE}, including Assumptions~\ref{assB} and \ref{assA}, we assume that~$\calA$ is self-adjoint and that~$f\colon [0, T]\times\V\to\V$ is two times Fr{\'e}chet differentiable in a strip along the exact solution~$u$ with uniformly bounded derivatives. Further we assume that the right-hand side~$g$ and~$u$ are sufficiently smooth, the latter with derivatives in~$\V$. Then, the approximation obtained by Algorithm~\ref{alg:secondOrd} is second-order accurate, i.e., 
\[
  \Vert u_n - u(t_n) \Vert_\V 
  \ \lesssim\ \tau^2.
\]
\end{theorem}
\begin{proof}
We reduce the procedure performed in Algorithm~\ref{alg:secondOrd} to the unconstrained case. For this, assume that $u_n = \uon + \calB^- g_n \in\V$ is given with~$\uon \in \Vo$ and that~$u^\text{Eul}_{n+1}$ denotes the outcome of a single Euler step, cf.~Algorithm~\ref{alg:expEuler}. By~$\uoEuler$ we denote the outcome of a Euler step for the unconstrained system 
\[
  \uodot(t) + \calAo\uo(t) 
  = \tilde f(t,\uo(t))
  \qquad\text{in }\Vo^*
\]
with~$\tilde f$ defined by~$\tilde f(t,\uo) := \iota_0\, [ f(t, \uo + \calB^- g(t)) - \calB^- \dot g(t)]$ and initial data~$\uon$. For this, we know that~$u^\text{Eul}_{n+1} = \uoEuler + \calB^- g_{n+1}$.  
By the given assumptions, it follows from~\cite[Th.~2.17]{HocO10} that 
\[
  \uonn 
  := \uoEuler + \tau\varphi_2(-\tau\calAo) \big[ \tilde f(t_{n+1},\uoEuler) - \tilde f(t_n, \uon) \big]
\]
defines a second-order approximation of~$\uo(t_{n+1})$. This in turn implies that 
\begin{align*}
  u_{n+1}
  &:= \uonn + \calB^-g_{n+1} \\
  &= u^\text{Eul}_{n+1} + \tau\varphi_2(-\tau\calAo) \big[ \tilde f(t_{n+1},\uoEuler) - \tilde f(t_n, \uon) \big] \\
  &= u^\text{Eul}_{n+1} + \tau\varphi_2(-\tau\calAo)\, \iota_0 \big[ f(t_{n+1},u^\text{Eul}_{n+1}) - \calB^- \dot g_{n+1} - f(t_n, u_n) + \calB^- \dot g_n \big]
\end{align*}
satisfies the error estimate
\[
  \Vert u_{n+1} - u(t_{n+1}) \Vert_\V 
  = \Vert \uonn - \uo(t_{n+1}) \Vert_\V 
  \lesssim \tau^2. 
\]
It remains to show that~$u_{n+1}$ is indeed the outcome of Algorithm~\ref{alg:secondOrd}. Following the construction in Section~\ref{sec:secondOrd:method}, we conclude that 
\[
  u_{n+1} = u^\text{Eul}_{n+1} + \varphi_0(-\tau\calAo)w_n'' - w_n'' + w_n'
\]
with~$w_n'$ and~$w_n''$ denoting the solutions of~\eqref{eqn:second:wnA} and~\eqref{eqn:second:wnB}, respectively. Finally, $\varphi_0(-\tau\calAo)w_n''$ is computed in line~\ref{alg:secondOrd:z} of Algorithm~\ref{alg:secondOrd} which completes the proof. 
\end{proof}
Up to now we have assumed that $f$ maps to $\V$, leading to the desired second-order convergence. In the following, we reconsider the more general case with~$f\colon [0,T] \times \calV \to \calH$. For PDEs it is well-known that the exponential integrator given by the Butcher tableau~\eqref{eqn:butcher_second_order} has a reduced convergence order if we assume~$\tfrac{\mathrm{d}^2}{\mathrm{d}t^2}  f(\cdot, u(\cdot)) \in L^\infty(0,T;\calH)$, cf.~\cite[Th.~4.3]{HocO05}. This carries over to the PDAE case. 
\begin{theorem}[Convergence under weaker assumptions on~$f$]
\label{thm:threehalfOrder}
Consider the assumptions from Theorem~\ref{thm:existence_uniqueness_PDAE} and let the step size~$\tau$ be sufficiently small such that the discrete solution~$u_n$ lies in a strip along~$u$, where~$f$ is locally Lipschitz continuous with a uniform constant~$L>0$. Further assume that~$g\in H^{3}(0,T;\Q^*)$. If the exact solution of~\eqref{eqn:PDAE} satisfies~$f(\, \cdot\, ,u(\, \cdot\, )) \in H^{2}(0,T;\cH)$, then the approximation~$u_n$ obtained by Algorithm~\ref{alg:secondOrd} satisfies the error bound 
\begin{align*}
 \Vert u_n - u&(t_n) \Vert_{\calV}^2\\
  \lesssim&\ \tau^3 \int_0^{t_n} \|\ddts f(t, u(t))\|^2_{\calH} + \| \calB^- \ddot g(t)\|^2_{\calH} \dt
  + \tau^4 \int_0^{t_n} \|\tfrac{\mathrm{d}^2}{\mathrm{d}t^2}  f(t, u(t))\|^2_{\calH} + \| \calB^- \dddot g(t)\|^2_{\calH}\dt.
\end{align*}
Note that the involved constant only depends on $t_n$, $L$, and the operator~$\calA$.
\end{theorem}
\begin{proof} 
Let~$U^{\text{Eul}}$ be the function constructed in the proof of Theorem~\ref{thm:Euler} which satisfies	$U^{\text{Eul}}(t_n) = u_n$ and $U^{\text{Eul}}(t_{n+1}) = u^{\text{Eul}}_{n+1}$ and set $U(t):= U^{\text{Eul}}(t) + z(t)-w_n^{\prime \prime} + \frac{t-t_n}{\tau}w_n^\prime$. This function satisfies 
\begin{equation*}
U(t_n)=U^{\text{Eul}}(t_n)=u_n,\qquad\quad 
U(t_{n+1})=U^{\text{Eul}}(t_{n+1})+z(t_{n+1}) - w_n^{\prime \prime} + w_n^\prime=u_{n+1}.
\end{equation*}
Note that the estimates~\eqref{eqn:estimate_Euler_0} and~\eqref{eqn:estimate_Euler_n} are still valid if one replaces~$u_{n+1}$ by $U^{\text{Eul}}(t_{n+1})$ on the left-hand side of these estimates. As in the proof of Theorem~\ref{thm:Euler}, we can interpret~$U$ as the solution of a PDAE on $[t_n,t_{n+1}]$. The corresponding right-hand sides are then given by
\begin{equation*}
f(t_n,u_n) 
+ \tfrac{t-t_n}{\tau}\,\big(f(t_{n+1},U^\text{Eul}(t_{n+1})) - f(t_n,u_n)\big) 
+ \calB^-\big(\dot g (t) - \dot g_n - \tfrac{t-t_n}{\tau}\, (\dot g_{n+1}  - \dot g_n)\big)
\end{equation*}
for the dynamic equation and~$g(t)$ for the constraint.  
Then, by Young's inequality, Gronwall's lemma, and error bounds for the Taylor expansion we get
\begin{align*}
\| 
u(t_{n+1}) - u_{n+1} \|^2_{\calA_1} 
&\leq e^{c\tau}\, \Big[ (1+4 \tau \tfrac{L^2}{\mu})\, \|u(t_{n}) - u_{n}\|^2_{\calA_1} + 4\, \tau \tfrac{L^2}{\mu}\, \|(u - U^{\text{Eul}})(t_{n+1})\|^2_{\calA_1}\\
&\qquad + \tau^4 \int_{t_n}^{t_{n+1}} \tfrac{2}{15}\, \| \tfrac{\mathrm{d}^2}{\mathrm{d}t^2} f(t,u(t))\|_{\calH}^2 + \tfrac{2}{45}\, \| \calB^- \dddot{g}(t) \|_{\calH}^2 \dt \Big]
\end{align*}
with $c=2c^2_{\calA_2}{\mu}^{-1}$. 
The stated error bound then follows by an iterative application of the previous estimate together with the estimates~\eqref{eqn:estimate_Euler_0}, \eqref{eqn:estimate_Euler_n} and the norm equivalence of~$\|\cdot\|_{\calV}$ and~$\|\cdot\|_{\calA_1}$.
\end{proof}
The actual performance of the proposed scheme is presented in the numerical experiments of Section~\ref{sec:numerics}. We close this section with remarks on alternative second-order schemes. 
%
%
\subsection{A class of second-order schemes}
The analyzed scheme~\eqref{eqn:butcher_second_order} is a special case of a one-parameter family of exponential Runge-Kutta methods described by the tableau 
\begin{equation*}
\begin{array}{r|cc}
0 &  &  \\
c_2 & \varphi_{1,2} & \\
\hline
& \varphi_1-\tfrac{1}{c_2} \varphi_2 & \tfrac{1}{c_2}\varphi_2
\end{array} 
\end{equation*}
with positive parameter~$c_2>0$, cf.~\cite{HocO10}. Therein, $\varphi_1$ stands for $\varphi_1(-\tau\calAo)$, whereas~$\varphi_{1,2}$ is defined by~$\varphi_1(-c_2\tau\calAo)$. Obviously, we regain~\eqref{eqn:butcher_second_order} for $c_2=1$. 

For $c_2\neq 1$, the resulting scheme for constrained systems calls for two additional saddle point problems in order to compute~$\calB^-g(t_n+c_2 \tau)$ and~$\calB^-\dot g(t_n+c_2 \tau)$. This then leads to an exponential integrator summarized in Algorithm~\ref{alg:secondOrd:B} with the abbreviations
\[
g_{n,2} := g(t_n+c_2 \tau),\qquad 
\dot g_{n,2} := \dot g(t_n+c_2 \tau),\qquad 
t_{n,2} := t_n + c_2\tau. 
\]
\begin{algorithm}
	\setstretch{1.1}
	\caption{A class of second-order exponential integrators}
	\label{alg:secondOrd:B}
	\begin{algorithmic}[1]
		\State {\bf Input}: step size~$\tau$, consistent initial data~$u_0\in\V$, right-hand sides $f$, $g$
		\vspace{0.5em}
		\For{$n=0$ {\bf to} $N-1$} 
		\State compute $\calB^-g_n$, $\calB^- g_{n,2}$, $\calB^- g_{n+1}$, $\calB^- \dot g_n$, $\calB^- \dot g_{n,2}$, and $\calB^- \dot g_{n+1}$  by~\eqref{eqn:Binvers}
		\State compute $w_n$ by~\eqref{eqn:wn}
		\State solve~\eqref{eqn:zn} on~$[t_n,t_{n,2}]$ with initial condition~$z(t_n)=u_n -\calB^-g_n - w_n$
		\State set $u_{n,2} = z(t_{n,2})+w_n + \calB^-g_{n,2}$
		\State compute~$w_n'$ by~\eqref{eqn:second:wnA} with right-hand side \[\tfrac{1}{c_2}\, \big(f(t_{n,2},u_{n,2})-f(t_{n},u_{n}) - \calB^-\dot g_{n,2} + \calB^-\dot g_{n}\big)\]  
		\State compute~$w_n''$ by~\eqref{eqn:second:wnB} 
		\State solve~\eqref{eqn:zn} on~$[t_n,t_{n+1}]$ with initial condition~$z(t_n)=u_n -\calB^-g_n - w_n+w_n''$ 
		\State set $u_{n+1} = z(t_{n+1}) + w_n +  w_n' - w_n'' + \calB^-g_{n+1}. $
		\EndFor 	
	\end{algorithmic}
\end{algorithm}

We emphasize that all convergence results of Theorems~\ref{thm:SecondOrdr} and~\ref{thm:threehalfOrder} transfer to this family of second-order integrators. In a similar manner, Runge-Kutta schemes of higher order may be translated to the here considered constrained case. 

%
\section{Numerical Examples}\label{sec:numerics}
In this final section we illustrate the performance of the introduced time integration schemes for two numerical examples. The first example is a heat equation with nonlinear dynamic boundary conditions. In the second experiment, we consider the case of a non-symmetric differential operator for which the theory is not applicable.  

Since exponential integrators for PDAEs are based on the exact solution of homogeneous systems of the form~\eqref{eqn:zn}, we first discuss the efficient solution of such systems. 
%
%
\subsection{Efficient solution of homogeneous DAEs with saddle-point structure} 
This subsection is devoted to the approximation of~$z(t)$, which is needed in line~\ref{alg:expEuler_z} of Algorithm~\ref{alg:expEuler} and in line~\ref{alg:secondOrd:z} of Algorithm~\ref{alg:secondOrd}. Given a spatial discretization, e.g., by a finite element method, the PDAE~\eqref{eqn:zn} turns into a DAE of index $2$, namely 
\begin{subequations}
\label{eqn:DAE}	
\begin{alignat}{5}
  M\dot{x}(t)&\ +\ & A x(t)&\ +\ & B^T \lambda(t)\, &= 0, \label{eqn:DAEa}\\
  & & B x(t)& & &= 0 \label{eqn:DAEb}
\end{alignat}
\end{subequations}
with consistent initial value~$x(0)=x_0$, $Bx_0 = 0$. The matrices satisfy~$M,A\in \R^{n \times n}$ and~$B \in \R^{m \times n}$ with $m\leq n$. Here, the mass matrix $M$ is symmetric, positive definite and $B$ has full rank. The goal is to find an efficient method to calculate the solution~$x$ at a specific time point $t\in [0,T]$. 

Let us first recall the corresponding ODE case. There exist various methods to approximate the solution~$x(t)=e^{-At} x_0$ of the linear ODE~$\dot{x}(t) + Ax(t)=0$ with initial condition $x(0)=x_0$, $A\in \R^{n\times n}$, for an overview see~\cite{MolV78}. This includes Krylov subspace methods to approximate the action of the matrix exponential~$e^{-At}$ to a vector, see \cite{Saa92, HocL97, EieE06}, but also methods based on an interpolation of~$e^{-At} x_0$ by Newton polynomials~\cite{CalO09}. 
The first approach is based on the fact that the solution $e^{-At} x_0 = \sum_{k=0}^\infty \frac{1}{k!} (-At)^k x_0$ is an element of the Krylov subspace 
\begin{equation*}
\calK_n:=\calK_n(-A,x_0):=\operatorname{span} \{ x_0,-Ax_0,\ldots,(-A)^{n-1}x_0\}.
\end{equation*}
Now, we approximate $e^{-At} x_0$ by an element of~$\calK_r$ with~$r$ relatively small compared to~$n$. For this, we generate an orthogonal basis of~$\calK_r$ using the Arnoldi algorithm with $v_1 = {x_0}/{\|x_0\|}$ as initial vector. This yields~$-V_r^T A V_r = H_r$ with an isometric matrix $V_r \in \R^{n\times r}$ and an upper Hessenberg matrix $H_r\in \R^{r\times r}$. Since the columns of $V_r$ are orthonormal and span~$\calK_r$, $H_r$ is the orthogonal projection of $-A$ onto $\calK_r$. Therefore, it is reasonable to use the approximation
\begin{equation*}
e^{-At}x_0 
\approx \|x_0\|\, V_r\,  e^{H_r t} e_1
\end{equation*}
with unit basis vector $e_1 \in \R^r$, cf.~\cite{HocL97}. We like to emphasize that the Arnoldi algorithm does not use the explicit representation of~$A$ but only its action onto a vector.

We return to the DAE case~\eqref{eqn:DAE}. By~\cite[Th.~2.2]{EmmM13} there exists a matrix~$X \in \R^{n\times n}$ such that the solution~$x$ of~\eqref{eqn:DAE} with arbitrary consistent initial value~$x_0$ is given by $x(t) = e^{Xt}x_0$. Furthermore, there exists a matrix-valued function $\Lambda\in C^\infty([0,\infty);\R^{m\times n})$ with $\lambda(t)=\Lambda(t)x_0$. To calculate the action of $X$ we note that by~\eqref{eqn:DAEb} also $B\dot x =0$ holds. 
We define~$y:=Xx_0$ and $\mu := \Lambda(0)x_0$. Then with equation~\eqref{eqn:DAEa}, $B\dot x =0$, and $t \to 0^+$ we get
\begin{subequations}
\label{eqn:saddlepoint_problem_Krylov}
\begin{alignat}{5}
  M y&\ +\ & B^T \mu \, &= -A x_0, \label{eqn:saddlepoint_problem_Krylov_a}\\
  B y& & &= 0. \label{eqn:saddlepoint_problem_Krylov_b}
\end{alignat}
\end{subequations}
Since the solution of~\eqref{eqn:saddlepoint_problem_Krylov} is unique, its solution~$y$ describes the action of $X$ applied to~$x_0$. As a result, we can approximate the solution of the DAE~\eqref{eqn:DAE} in an efficient manner by using $x(t)=e^{Xt}x_0$, the saddle point problem~\eqref{eqn:saddlepoint_problem_Krylov}, and Krylov subspace methods. For the numerical experiments we have adapted the code provided in~\cite{NieW12}. 
\begin{remark}
Given an approximation~$x_t \approx x(t)$, the solution $\mu$ of~\eqref{eqn:saddlepoint_problem_Krylov} with right-hand side~$-A x_t$ provides an approximation of the Lagrange multiplier~$\lambda(t)$. 
\end{remark}
\begin{remark}
Since the saddle point problem~\eqref{eqn:saddlepoint_problem_Krylov} has to be solved several times in every time step, the numerical solution~$\tilde{x}$ of~\eqref{eqn:DAE} may not satisfy the constraint~\eqref{eqn:DAEb} due to round-off errors. To prevent a drift-off,  one can project~$\tilde{x}$ onto the kernel of~$B$ -- by solving an additional saddle point problem. 
\end{remark}

\subsection{Nonlinear dynamic boundary conditions}
In this first example we revisit Example~\ref{exp:dynBC} and consider the linear heat equation with nonlinear dynamic boundary conditions, cf.~\cite{SprW10}. More precisely, we consider the system 
\begin{subequations}
\label{eqn:dynBC}
\begin{align}
  \dot  u - \kappa\, \Delta u &= 0 \phantom{f_\Gamma(t,u)} 
  \quad\text{in } \Omega:=(0,1)^2, \label{eqn:dynBC:a} \\
  \dot u + \partial_n u + \alpha\, u  &= f_\Gamma(t,u) \phantom{0}
  \quad \text{on }	\Gamma_\text{dyn} := (0,1) \times \{0\} \label{eqn:dynBC:b} \\
  u &= 0 \phantom{f_\Gamma(t,u)} 
  \quad\text{on } \Gamma_D := \partial\Omega \setminus \Gamma_\text{dyn}  \label{eqn:dynBC:c}
\end{align}
\end{subequations}
with $\alpha=1$, $\kappa=0.02$, and the nonlinearity~$f_\Gamma(t,u)(x) = 3\cos(2\pi t) - \sin(2\pi x) - u^3(x)$. As initial condition we set~$u(0) = u_0 = \sin(\pi x) \cos(5\pi y/2)$. Following~\cite{Alt19}, we can write this in form of a PDAE, namely as 
\begin{subequations}
\label{eqn:dynBC:PDAE}
\begin{align}
	\begin{bmatrix} \dot u \\ \dot p  \end{bmatrix}
	+ \begin{bmatrix} \calK &  \\  & \alpha\, \end{bmatrix}
	\begin{bmatrix} u \\ p  \end{bmatrix}
	+ \calB^*\lambda 
	&= \begin{bmatrix} 0 \\ f_\Gamma(t, p) \end{bmatrix} \qquad \text{in } \V^*, \\
	\calB\, \begin{bmatrix} u \\ p  \end{bmatrix} \phantom{i + \calB \lambda} &= \phantom{[]} 0\hspace{5.55em} \text{in } \Q^* \label{eqn:dynBC:PDAE:b}
\end{align}
\end{subequations}
with spaces $\V = H^1_{\Gamma_D}(\Omega) \times H^{1/2}(\Gamma_\text{dyn})$, $\cH = L^2(\Omega) \times L^2(\Gamma_\text{dyn})$, $\Q = [H^{1/2}(\Gamma_\text{dyn})]^*$ and constraint operator $\calB(u,p)=u|_{\Gamma_\text{dyn}} - p$. Here, $p$ denotes a dummy variable modeling the dynamics on the boundary~$\Gamma_\text{dyn}$. The constraint~\eqref{eqn:dynBC:PDAE:b} couples the two variables~$u$ and~$p$. 
This example fits into the framework of this paper with~$g\equiv 0$. Further, the nonlinearity satisfies the assumptions of the convergence results in Theorems~\ref{thm:Euler} and~\ref{thm:threehalfOrder} due to well-known Sobolev embeddings, see~\cite[p.~17f]{Rou05}. 

For the simulation we consider a spatial discretization by bilinear finite elements on a uniform mesh with mesh size~$h=1/128$. The initial value of $p$ is chosen in a consistent manner, i.e., by $u_0|_{\Gamma_\text{dyn}}$. An illustration of the dynamics is given in Figure~\ref{fig:dynBC:up}. 
The convergence results of the exponential Euler scheme of Section~\ref{sec:Euler} and the exponential integrator introduced in Section~\ref{sec:secondOrd} are displayed in Figure~\ref{fig:dynBC:conv} and show first and second-order convergence, respectively. Note that the theory only ensures an order $\sfrac{3}{2}$ for the second integrator. However, the spatial discretization acts as a regularization in space, i.e., we expect order $\sfrac{3}{2}$ only for~$h\to 0$ and non-smooth initial data. 

Finally, we note that the computations remain stable for very coarse step sizes~$\tau$, since we do not rely on a CFL condition here. 
%
\definecolor{colInitial}{RGB}{113,104,238} 
\definecolor{colP}{RGB}{189,26,26} 
\begin{figure}
\includegraphics{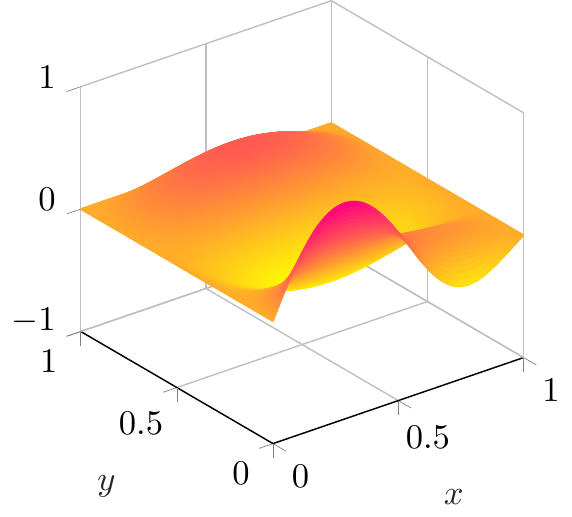}
\input{pics/dynBC_p_h76_tau001}
\caption{Illustration of the solution $(u, p)$. The left figure shows~$u$ at time $t=0.7$, whereas the right figure includes several snapshots of~$p$ in the time interval~$[0, 0.7]$. The dashed line shows the initial value of~$p$. Both results are obtained for mesh size~$h=1/128$ and step size~$\tau = 1/100$.
}
\label{fig:dynBC:up}
\end{figure}
%
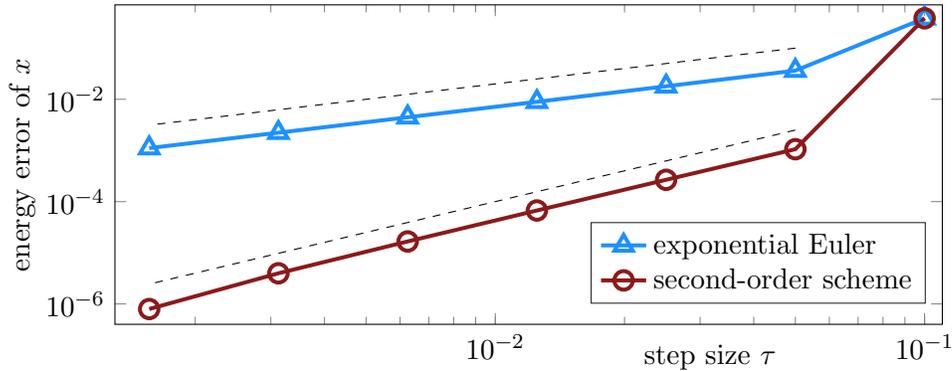
\begin{figure}
%
%
\begin{tikzpicture}

\begin{axis}[%
width=11.0cm,
height=4.25cm,
at={(0.758in,0.481in)},
scale only axis,
xmode=log,
xmin=0.0013,
xmax=0.11,
xminorticks=true,
xlabel style={below=-0.9mm, right=1.4cm},
ymode=log,
ymin=4e-07,
ymax=0.7,
xlabel=step size $\tau$,
ylabel=energy error of $x$,
yminorticks=true,
axis background/.style={fill=white},
legend style={legend cell align=left, align=left, draw=white!15!black, at={(0.98,0.32)}}
]
\addplot [color=myBlue3, mark=triangle, mark size=4pt, line width=1.5pt]
  table[row sep=crcr]{%
0.1	0.38253811462875\\
0.05	0.0363270323489102\\
0.025	0.0179574531005322\\
0.0125	0.00892824264638325\\
0.00625	0.00445177192981783\\
0.003125	0.00222292004597674\\
0.0015625	0.00111080741637044\\
};
\addlegendentry{exponential Euler}

\addplot [color=myRed3, mark=o, mark size=3.5pt, line width=1.5pt]
  table[row sep=crcr]{%
0.1	0.38253811462875\\
0.05	0.00105744574100574\\
0.025	0.000267045347071342\\
0.0125	6.70901376784542e-05\\
0.00625	1.66535487464527e-05\\
0.003125	3.97558064005907e-06\\
0.0015625	7.96205270439963e-07\\
};
\addlegendentry{second-order scheme}

\addplot [color=black, dashed]
table[row sep=crcr]{%
	0.05	0.1\\
	0.0016   0.0032\\
};

\addplot [color=black, dashed]
table[row sep=crcr]{%
	0.05	0.0025\\
	0.0016   0.00000256\\
};

\end{axis}
\end{tikzpicture}%
\caption{Convergence history for the error in~$x=[u; p]$, measured in the (discrete) $\A$-norm. The dashed lines show first and second-order rates. }
\label{fig:dynBC:conv}
\end{figure}
%
%
\subsection{A non-symmetric example}
In this final example we consider a case for which Assumption~\ref{assA} is not satisfied. More precisely, we consider the coupled system 
%
\begin{align*}
	\dot u - \partial_{xx} u - \partial_{xx} v &= -u^3 \qquad\text{in } (0,1), \\
	\dot v + u - \partial_{xx} v  &= -v^3 \hspace{0.04em}\qquad \text{in } (0,1)
\end{align*}
%
with initial value 
$$u_0(x) = v_0(x) = \sum_{k=1}^\infty \frac{\sin(k \pi x)}{k^{1.55}}$$
 and the constraint $u(t,1)-v(t,1) = g(t) = e^{2t}-1$. At the other boundary point~$x=0$ we prescribe homogeneous Dirichlet boundary conditions. In this example, the operator $\calA$ has the form $-[ \partial_{xx},\, \partial_{xx};\, - \id,\, \partial_{xx}]$. Thus, the non-symmetric part~$\calA_2$ includes a second-order differential operator which contradicts Assumption~\ref{assA}. As a consequence, non of the convergence results in this paper apply.

The numerical results are shown in Figure~\ref{fig:nonsym:conv}, using a finite element discretization with varying mesh sizes~$h$. One can observe that the exponential Euler scheme still converges with order $1$, whereas the second-order scheme introduced in Section~\ref{sec:secondOrd} clearly converges with a reduced rate. Moreover, the rate decreases as the mesh size~$h$ gets smaller. 
By linear regression one can approximate the convergence rate as a value between~$1.40$ (coarsest mesh, $h=1/32$) and $1.34$ (finest mesh, $h=1/256$). Thus, the convergence rate is strictly below~$\sfrac 3 2$. 
%
\begin{figure}
%
%
\definecolor{myColor1}{rgb}{0.89, 0.88, 0.11}%
\definecolor{myColor2}{rgb}{0.90, 0.61, 0.15}%
\definecolor{myColor3}{rgb}{0.80, 0.25, 0.12}%
\definecolor{myColor4}{rgb}{0.61, 0.11, 0.03}%
\begin{tikzpicture}

\begin{axis}[%
width=11.0cm,
height=5.75cm,
at={(0.747in,0.477in)},
scale only axis,
xmode=log,
xmin=8e-05,
xmax=0.005,
xminorticks=true,
xlabel style={below=-0.9mm, right=-2.0cm},
ymode=log,
ymin=2e-06,
ymax=9e-03,
xlabel=step size $\tau$,
ylabel=$H^1$-error of $x$,
yminorticks=true,
axis background/.style={fill=white},
legend style={legend cell align=left, align=left, draw=white!15!black, at={(0.98,0.42)}}
]
\addplot[color=myColor1, line width=1.5pt]
table[row sep=crcr]{%
0.05	0.0799793403789066\\
0.025	0.0439617033825785\\
};
\addlegendentry{$h=1/32$}
\addplot[color=myColor2, line width=1.5pt]
table[row sep=crcr]{%
0.05	0.0802008665291415\\
0.025	0.0441254206556253\\
};
\addlegendentry{$h=1/64$}
\addplot [color=myColor3, line width=1.5pt]
table[row sep=crcr]{%
0.05	0.0802388511933204\\
0.025	0.0441558789092972\\
};
\addlegendentry{$h=1/128$}
\addplot[color=myColor4, line width=1.5pt]
table[row sep=crcr]{%
0.05	0.0802437669449783\\
0.025	0.0441620999329646\\
};
\addlegendentry{$h=1/256$}

\addplot[color=myColor1, mark=triangle,mark size=4pt, line width=1.5pt]
  table[row sep=crcr]{%
0.05	0.0799793403789066\\
0.025	0.0439617033825785\\
0.0125	0.0208048429373345\\
0.00625	0.00902710583471356\\
0.003125	0.00371418298993378\\
0.0015625	0.0017607974420459\\
0.00078125	0.000861324080293846\\
0.000390625	0.00042638794805749\\
0.0001953125	0.000212163045361146\\
9.765625e-05	0.000105829156566564\\
};

\addplot[color=myColor2, mark=triangle,mark size=4pt, line width=1.5pt]
  table[row sep=crcr]{%
0.05	0.0802008665291415\\
0.025	0.0441254206556253\\
0.0125	0.020920683365349\\
0.00625	0.00911044786187226\\
0.003125	0.0037335910506488\\
0.0015625	0.00176879192800965\\
0.00078125	0.000864938215926925\\
0.000390625	0.000428069918799447\\
0.0001953125	0.000212980086252099\\
9.765625e-05	0.000106232105218543\\
};

\addplot [color=myColor3, mark=triangle,mark size=4pt, line width=1.5pt]
  table[row sep=crcr]{%
0.05	0.0802388511933204\\
0.025	0.0441558789092972\\
0.0125	0.0209402392217409\\
0.00625	0.00912226978441421\\
0.003125	0.0037415234852887\\
0.0015625	0.00177025048264824\\
0.00078125	0.000865629091580593\\
0.000390625	0.000428397993490292\\
0.0001953125	0.000213141203045606\\
9.765625e-05	0.000106311146925543\\
};

\addplot[color=myColor4, mark=triangle, mark size=4pt, line width=1.5pt]
  table[row sep=crcr]{%
0.05	0.0802437669449783\\
0.025	0.0441620999329646\\
0.0125	0.0209441304079959\\
0.00625	0.00912421750759777\\
0.003125	0.00374245006551437\\
0.0015625	0.00177054855887525\\
0.00078125	0.000865772890189205\\
0.000390625	0.000428467222665489\\
0.0001953125	0.000213175598263249\\
9.765625e-05	0.000106328162044826\\
};

\addplot [color=myColor1, mark=o, mark size=3.5pt, line width=1.5pt]
  table[row sep=crcr]{%
0.05	0.0194293386936316\\
0.025	0.00942799552482282\\
0.0125	0.00419454340054839\\
0.00625	0.00173158174828716\\
0.003125	0.000675570670020443\\
0.0015625	0.000249218754801343\\
0.00078125	8.64137005415434e-05\\
0.000390625	2.89514066082455e-05\\
0.0001953125	1.05831643476437e-05\\
9.765625e-05	3.5205599930691e-06\\
};

\addplot[color=myColor2, mark=o, mark size=3.5pt, line width=1.5pt]
  table[row sep=crcr]{%
0.05	0.0194896271130602\\
0.025	0.00947280275813653\\
0.0125	0.00422712569625351\\
0.00625	0.00175796859398999\\
0.003125	0.000698034554833445\\
0.0015625	0.00026774453295176\\
0.00078125	9.92974476831029e-05\\
0.000390625	3.534947590794e-05\\
0.0001953125	1.19879617842484e-05\\
9.765625e-05	3.97711674248949e-06\\
};

\addplot [color=myColor3, mark=o, mark size=3.5pt, line width=1.5pt]
  table[row sep=crcr]{%
0.05	0.0195018461078098\\
0.025	0.00948089770948322\\
0.0125	0.00423184739705542\\
0.00625	0.00176104445627439\\
0.003125	0.000700369711508355\\
0.0015625	0.000269889461863043\\
0.00078125	0.000101418097653568\\
0.000390625	3.72919987859201e-05\\
0.0001953125	1.3393e-05\\
9.765625e-05	4.6603e-06\\
};

\addplot[color=myColor4, mark=o, mark size=3.5pt, line width=1.5pt]
  table[row sep=crcr]{%
0.05	0.019504687860311\\
0.025	0.00948265980627946\\
0.0125	0.0042326814071096\\
0.00625	0.00176142074225775\\
0.003125	0.000700513499197293\\
0.0015625	0.000269959142820367\\
0.00078125	0.000101504240897599\\
0.000390625	3.74269703520746e-05\\
0.0001953125	1.35826622223156e-05\\
9.765625e-05	4.86451118182703e-06\\
};

\addplot [color=black, dashed] 
  table[row sep=crcr]{%
0.05	0.0375\\
1e-04	7.5e-05\\
};

\addplot [color=black, dashed] 
  table[row sep=crcr]{%
0.05	0.0559016994\\ 
1e-04	1.58113883e-06\\ 
};
\end{axis}
\end{tikzpicture}%
\caption[~]{Convergence history for the error in~$x=[u; v]$, measured in the (discrete) $H^1(0,1)$-norm, including Dirichlet boundary conditions in $x=0$. The graphs show the results of the exponential Euler scheme~(\tikz{\node[mark size=3.5pt,line width=1.5pt] at (0,0) {\pgfuseplotmark{triangle}};}) and the second order scheme (\tikz{ \node[mark size=3pt,line width=1.5pt] at (0,0) {\pgfuseplotmark{o}};}) for different values of~$h$, displayed by its color. The dashed lines illustrate the orders $1$ and~$\sfrac 3 2$, respectively.}
\label{fig:nonsym:conv}
\end{figure}
A deeper analysis with fractional powers of~$\calA$ may predict the exact convergence rate, cf.~\cite[Th.~4.2 \& Th.~4.3]{HocO05}. However, this is a task for future work.
%
%
\section{Conclusion}\label{sec:conclusion}
In this paper, we have introduced a novel class of time integration schemes for semi-linear parabolic equations restricted by a linear constraint. For this, we have combined exponential integrators for the dynamical part of the system with (stationary) saddle point problems for the 'algebraic part' of the solution. As a result, we obtain exponential integrators for constrained systems of parabolic type for which we have proven convergence of first and second order, respectively. Numerical experiments illustrate the obtained results.
%
%
\newcommand{\etalchar}[1]{$^{#1}$}

\end{document}